\documentclass[a4paper,12pt,intlimits,oneside]{amsart}

\usepackage{amssymb,latexsym,amsmath,mathrsfs}
\usepackage{amsfonts,amsbsy,bm}
\usepackage{color}
\usepackage[margin=2cm]{geometry}

\usepackage{amsmath}
\usepackage{amssymb}
\usepackage{amsfonts}
\usepackage{amsthm,amscd}

\newtheorem{theorem}{Theorem}[section]
\newtheorem{proposition}[theorem]{Proposition}
\newtheorem{corollaire}[theorem]{Corollary}
\newtheorem{lemme}[theorem]{Lemma}

\newtheorem{remark}[theorem]{Remark}

\theoremstyle{definition}
\newcommand{\comment}[1]{}

\numberwithin{equation}{section}

\newcommand{\epf}{ $\Box$\medskip}

\def\supp{{\rm supp}}

\theoremstyle{definition}

\begin{document}
\title[Factorization and Atomic Decomposition]{Factorization and Atomic Decomposition in Hardy-Orlicz Spaces on the Upper Half-Plane with Applications to Hankel Operators}
\author[J.M. Tanoh Dje]{Jean$-$Marcel Tanoh Dje}
\address{Unit\'e de Recherche et d'Expertise Num\'erique, Universit\'e Virtuelle de C\^ote d'Ivoire, Cocody II-Plateaux - 28 BP 536 ABIDJAN 28}
\email{{\tt tanoh.dje@uvci.edu.ci}}
\author[J. Feuto]{Justin Feuto}
\address{Laboratoire de Math\'ematiques et Applications, UFR Math\'ematiques$-$Informatique, Universit\'e F\'elix Houphou\"et-Boigny Abidjan-Cocody, 22 B.P 1194 Abidjan 22. C\^ote d'Ivoire}
\email{{\tt justfeuto@yahoo.fr}}
\author[Beno\^it Sehba]{Beno\^it F. Sehba}
\address{Department of Mathematics, University of Ghana,\\ P. O. Box LG 62 Legon, Accra, Ghana}
\email{bfsehba@ug.edu.gh}


\date{}

\begin{abstract}
In this work, we establish a strong factorization for Hardy-Orlicz spaces on the upper half-plane. We show that the product of two functions belonging respectively to Hardy-Orlicz spaces $H^{\Phi_{1}}$ and $H^{\Phi_{2}}$ lies in a third space $H^{\Phi_{3}}$, and every holomorphic function in 
$H^{\Phi_{3}}$ admits such a decomposition. We then provide an atomic decomposition for certain Hardy-Orlicz spaces, which allows us to describe the topological dual of these spaces when the associated function is concave. Finally, these results are applied to the study of the continuity of the Hankel operators in this setting.
\end{abstract}

\maketitle

\section{Introduction and statement of main results.}

Recall that a function $\Phi : [0,\infty) \rightarrow [0,\infty)$ is called an growth function if it is nondecreasing, $\displaystyle\lim_{t \to 0} \Phi(t) =\Phi(0)= 0$, $\Phi(t) > 0$
for $t \in (0,\infty)$ and $\displaystyle\lim_{t \to \infty} \Phi(t) = \infty$.  The growth function  $\Phi$ is said to be of upper type (resp. lower type) if there exists  $p \in (0,\infty)$ and a constant $C>1$ such that for all $t \in  [1,\infty)$ (resp. $t \in  [0,1]$) and $s \in  [0,\infty)$,
\begin{equation}\label{eq:sui8n}
\Phi(st)\leq Ct^{p}\Phi(s).\end{equation}
We denote by $\mathscr{U}^{p}$ (resp. $\mathscr{L}_{p}$) the set of all growth functions of upper-type $p \geq 1$ (resp. lower-type $0< p\leq 1$) such that the function 
$t\mapsto \frac{\Phi(t)}{t}$ is non decreasing (resp. non-increasing) on $(0,\infty)$. We put
$   \mathscr{U}:=\bigcup_{p\geq 1}\mathscr{U}^{p}$ (resp. $\mathscr{L}:=\bigcup_{0< p\leq 1}\mathscr{L}_{p}$).

\medskip

Two positive functions  $\Phi_{1}$ and $\Phi_{2}$ on $[0,\infty)$ are said to be equivalent (or $\Phi_{1} \sim \Phi_{2}$) if there exists a constant $c > 0$ such that
\begin{equation}\label{eq:equivalent}
c^{-1}\Phi_{1}(c^{-1}t) \leq \Phi_{2}(t)\leq c\Phi_{1}(ct), ~~ \forall~ t > 0.\end{equation} 

\medskip

We will assume in the sequel that  any element $\Phi$ of $\mathscr{U}$ (resp. $\mathscr{L}$) belongs to $\mathscr{C}^{1}([0,\infty))$ and is  
convex (resp. concave), 
(see for example \cite{djefeuto1, djesehb,  sehbaedgc}). 

\medskip

Let $\Phi$ be a growth function of lower type. Recall that the Hardy--Orlicz space on the unit disc 
 $\mathbb D := \left\lbrace z \in\mathbb C : \vert z\vert < 1\right\rbrace $ denoted by $H^\Phi(\mathbb D)$, is defined as the space of all holomorphic functions $G$ on $\mathbb{D}$ such that
$$   \|G\|_{H^{\Phi}(\mathbb D)}^{lux}:=\sup_{0\leq r<1}\inf\left\{\lambda>0 :  \frac{1}{2\pi}\int^{2\pi}_0\Phi\left(\frac{|G(re^{it})|}{\lambda}\right)dt \leq 1  \right\}< \infty. $$
Recently, the first two authors extended in their work \cite[Theorem 3]{djefeuto1} the strong factorization result originally established by Riesz for classical Hardy spaces (corresponding to $\Phi(t)=t^p$) to the setting of Hardy-Orlicz spaces on the complex unit disk. This result can be reformulated as follows: Let $\Phi_{1}$,  $\Phi_{2}$ and $\Phi_{3}$ be growth functions of the lower type. If  $\Phi_{3}^{-1} \sim \Phi_{1}^{-1}\cdot\Phi_{2}^{-1}$ then  $H^{\Phi_{3}}=H^{\Phi_{1}}\cdot H^{\Phi_{2}}$. 

This factorization result for Hardy-Orlicz spaces subsequently enabled the authors to establish both loss and gain estimates for Hankel operators acting between Hardy-Orlicz spaces (see respectively \cite[Theorem 4 ]{djefeuto1} and  \cite[Theorem 5 ]{djefeuto1}).

\medskip

The aim of this work is to carry out a similar study for Hardy-Orlicz spaces on the complex upper half-plane.

\medskip

Let  $\Phi$ be a lower type growth function.  The Orlicz space on $\mathbb{R}$ is the space  $L^{\Phi}(\mathbb{R})$ of  measurable functions $f : \mathbb{R} \longrightarrow \mathbb{C}$ which satisfy
$$ \|f\|_{L^{\Phi}(\mathbb{R})}^{lux}:=\inf\left\{\lambda>0 :  \int_{-\infty}^{\infty}\Phi\left(\frac{|f(t)|}{\lambda}\right)dt \leq 1  \right\}< \infty. $$
The Hardy-Orlicz space on  the upper complex half-plane  $\mathbb{C}_{+}:=\left\{ z=x+iy \in \mathbb{C} :   y > 0 \right\}$ is the space  $H^{\Phi}(\mathbb{C}_{+})$ of holomorphic functions $F$ on  $\mathbb{C}_{+}$ 
which satisfy
$$ \|F\|_{H^{\Phi}(\mathbb{C}_{+})}^{lux}:=\sup_{y>0 }\|F(\cdot+iy)\|_{L^{\Phi}(\mathbb{R})}^{lux}< \infty.   $$

In the following result, we present our statement of the strong factorization for Hardy-Orlicz spaces on the upper half-plane.

\begin{theorem}\label{pro:main2aqop}
Let $\Phi_{1}, \Phi_{2}, \Phi_{3} \in \mathscr{L} \cup \mathscr{U}$  such that $\Phi_{3}^{-1} \sim \Phi_{1}^{-1}\cdot \Phi_{2}^{-1}$, where $\Phi_{j}^{-1}$ denotes the inverse function of $\Phi_{j}$ for each $j \in \{1,2,3\}$. Then, for any $F_{1} \in H^{\Phi_{1}}(\mathbb{C}_{+})$ and $F_{2} \in H^{\Phi_{2}}(\mathbb{C}_{+})$, the product $F_{1}\cdot F_{2}$ belongs to $H^{\Phi_{3}}(\mathbb{C}_{+})$.  Conversely, for any $F \in H^{\Phi_{3}}(\mathbb{C}_{+})$, there exist functions $F_{1} \in H^{\Phi_{1}}(\mathbb{C}_{+})$ and $F_{2} \in H^{\Phi_{2}}(\mathbb{C}_{+})$ such that
$F = F_{1}\cdot F_{2}$.
Moreover, the following norm equivalence holds:
\begin{equation}\label{eq:ap21}
\|F\|_{H^{\Phi_{3}}(\mathbb{C}_{+})}^{\mathrm{lux}} \approx \|F_{1}\|_{H^{\Phi_{1}}(\mathbb{C}_{+})}^{\mathrm{lux}} \cdot \|F_{2}\|_{H^{\Phi_{2}}(\mathbb{C}_{+})}^{\mathrm{lux}}.
\end{equation}
\end{theorem}

When  $\Phi_{j}(t)=t^{p_{j}}$ with  $0< p_{j}  < \infty$ and $1/p_{3}= 1/p_{1} + 1/p_{2}$, Theorem \ref{pro:main2aqop} coincides with a classical result by Riesz for Hardy spaces (see \cite{wrudin}). 

\medskip

To establish the strong factorization of Hardy--Orlicz spaces on the complex upper half-plane, that is, to prove Theorem~\ref{pro:main2aqop}, we rely on the same strategy and arguments as those developed on the unit disc by the first two authors in \cite{djefeuto1}. We establish an inner-outer functions factorization of functions in the holomorphic Hardy-Orlicz spaces that we then use to prove the above result.

As in the unit disc setting, a natural application of this factorization result is the control of the Hankel operator, which will be also investigated in this work. For this, we introduce few needed spaces, and study some properties of these spaces and Hardy-Orlicz spaces of the real line and transformations acting from these spaces to their holomorphic analogues.

\medskip

We denote by $\mathscr{S}(\mathbb{R})$ the Schwartz space on $\mathbb{R}$ and by $\mathscr{S}'(\mathbb{R})$ its dual space, namely the space of tempered distributions.

\medskip

A distribution $f \in \mathscr{S}'(\mathbb{R})$ is said to be \emph{bounded} if, for every $\varphi \in \mathscr{S}(\mathbb{R})$, the convolution $\varphi \ast f$ belongs to $L^{\infty}(\mathbb{R})$.

\medskip

Let $m \in \mathbb{N}$. For $f \in \mathscr{S}'(\mathbb{R})$, the non-tangential grand maximal function $f_m^*$ is defined, for all $x \in \mathbb{R}$, by
\[
f_m^*(x) := \sup_{\varphi \in \mathcal{S}_m(\mathbb{R})} \sup_{t>0} \sup_{|y-x|<t} \left| \varphi_t \ast f(y) \right|,
\]
where, for $t>0$, $\varphi_t(\cdot) := \frac{1}{t}\varphi\!\left(\frac{\cdot}{t}\right)$ and
\[
\mathcal{S}_m(\mathbb{R}) := \left\{ \varphi \in \mathscr{S}(\mathbb{R}) : 
\mathcal{N}_m(\varphi) := \sup_{x \in \mathbb{R}} \sup_{0 \leq k \leq m+1} (1+|x|)^{2(m+2)} |\partial^k \varphi(x)| \leq 1 \right\}.
\]

\medskip

Let $\Phi \in \mathscr{L}_p$ with $0 < p \leq 1$. The Hardy--Orlicz space on $\mathbb{R}$, denoted by $H^\Phi(\mathbb{R})$, is defined as the set of all $f \in \mathscr{S}'(\mathbb{R})$ such that $f_{m_\Phi}^* \in L^\Phi(\mathbb{R})$. In this case, we define
\[
\|f\|_{H^\Phi(\mathbb{R})}^{\mathrm{lux}} := \left\| f_{m_\Phi}^* \right\|_{L^\Phi(\mathbb{R})}^{\mathrm{lux}},
\]
where the integer $m_\Phi$ is given by
\begin{equation}\label{eq:dqtap2}
m_\Phi := \lfloor 1/p \rfloor - 1,
\end{equation}
and $\lfloor \cdot \rfloor$ denotes the integer part.

\medskip

In what follows, whenever $\Phi \in \mathscr{L}$, the quantity $m_\Phi$ always refers to the definition given in \eqref{eq:dqtap2}.

\medskip

Let $s \in \mathbb{N}$ and $\Phi \in \mathscr{L}$. A measurable function $a$ is called a \emph{$(\Phi,s)$-atom} if it satisfies the following conditions:
\begin{itemize}
\item[(i)] there exists a finite interval $I \subset \mathbb{R}$ such that $\mathrm{supp}(a) \subset I$;
\item[(ii)] $a \in L^2(\mathbb{R})$ and $\|a\|_{L^2} \leq |I|^{1/2}\,\Phi^{-1}\!\left(\frac{1}{|I|}\right)$;
\item[(iii)] $\displaystyle \int_{\mathbb{R}} a(x)\,x^k\,dx = 0$ for all $0 \leq k \leq s$.
\end{itemize}

\medskip

The orthogonal projection from $L^2(\mathbb{R})$ onto $H^2(\mathbb{C}_+)$ is called the \emph{Szeg\"o projection} and is denoted by $\mathcal{P}_S$. It is given by
\[
\mathcal{P}_S(f)(z) := \frac{1}{2\pi i} \int_{\mathbb{R}} \frac{f(t)}{t - z}\, dt, \quad \forall z \in \mathbb{C}_+.
\]

\medskip

The following result provides an atomic decomposition of Hardy--Orlicz spaces on the complex upper half-plane.

\begin{theorem}\label{pro:main6maqmq4}
Let $\Phi \in \mathscr{L}$ and let $F$ be an analytic function on $\mathbb{C}_+$. The following assertions are equivalent:
\begin{itemize}
\item[(i)] $F \in H^\Phi(\mathbb{C}_+)$;
\item[(ii)] there exist a bounded tempered distribution $f$ on $\mathbb{R}$ and a sequence of $(\Phi,s)$-atoms $(a_j)_j$, with $s \in \mathbb{N}$ satisfying $s \geq m_\Phi$, such that
\[
f = \sum_{j=0}^\infty a_j \quad \text{in } H^\Phi(\mathbb{R}),
\]
and, for all $z \in \mathbb{C}_+$,
\[
F(z) = \mathcal{P}_S(f)(z) = \sum_{j=0}^\infty \mathcal{P}_S(a_j)(z).
\]
\end{itemize}
Moreover,
\[
\|F\|_{H^\Phi(\mathbb{C}_+)}^{\mathrm{lux}} \approx \|f\|_{H^\Phi(\mathbb{R})}^{\mathrm{lux}}
\]
and
\[
\lim_{y \to 0} \|F(\cdot + i y) - f\|_{H^\Phi(\mathbb{R})}^{\mathrm{lux}} = 0.
\]
\end{theorem}

To prove the above result, we first rely on the atomic decomposition of the Hardy--Orlicz space on $\mathbb{R}$, $H^\Phi(\mathbb{R})$ (see \cite{yangLiangKy}). Then, we use the fact that the Szeg\"o projection $\mathcal{P}_S$ is continuous and bounded from $H^\Phi(\mathbb{R})$ into $H^\Phi(\mathbb{C}_{+})$ in order to obtain the atomic decomposition of $H^\Phi(\mathbb{C}_{+})$.
We note that atomic decompositions of Hardy--Orlicz spaces have already been obtained by several authors in the case of the unit ball of $\mathbb{C}^n$ (see \cite{BoGre}, \cite{BoGreseh}). In these works, they are typically used to establish weak factorization results in Hardy--Orlicz spaces.

\medskip

 W wille exploit the above atomic decomposition of Hardy--Orlicz spaces to characterize the topological dual of the Hardy--Orlicz space \(H^{\Phi}(\mathbb{C}_{+})\) on the complex upper half-plane in the case where the growth function \(\Phi\) is concave. This will be then combined with the strong factorization of Hardy-Orlicz spaces (Theorem~\ref{pro:main2aqop}) to characterize bounded Hankel operators between different Hardy--Orlicz spaces.

\medskip

A locally integrable function $f$ on $\mathbb{R}$ is said to belong to the  $\mathrm{BMO}(\mathbb{R})$ space if
\begin{equation}\label{eq:dqtmqaaqa2}
\|f\|_{\mathrm{BMO}}:= \sup_{I \subset \mathbb{R}} \frac{1}{|I|}\int_{I}|f(x)-m_{I}(f)|dx  < \infty,
\end{equation}
where  the supremum is taking on all intervals $I \subset \mathbb{R}$ and  $m_{I}(f):= \frac{1}{|I|}\int_{I}f(s)ds$.  Here, for any
   measurable set $E$, $|E|$ denotes the Lebesgue measure of $E$. We emphasize that \(f\in \mathrm{BMO}(\mathbb{R})\) represents an equivalence class and $f=g$ if and only if $f-g$ is constant. Every function \(f\in \mathrm{BMO}(\mathbb{R})\) is Poisson integrable and therefore belongs to the space of tempered distributions \(\mathcal S'\).   
  It is well known that \(\mathrm{BMO}(\mathbb{R})\) is isomorphic to the dual space of \(H^{1}(\mathbb{R})\): for every bounded linear functional $\Lambda \in \bigl(H^{1}(\mathbb{R})\bigr)^{*}$, there exists a unique function $g\in \mathrm{BMO}(\mathbb{R})$ such that
\[
\Lambda(f)
=
\langle f,g\rangle
:=
\int_{\mathbb{R}}
f(x)g(x)\,dx,
\]
for every \(f\in H^{1}(\mathbb{R})\).

An analytic function $F$ on $\mathbb{C}_{+}$ belongs to  $BMOA(\mathbb{C}_{+})$ if and only if 
 $F$ can be written as
\begin{equation}\label{eq:dqtmqama2}
 F(x+iy) =P_{y}\ast f (x),
\end{equation}
where $f$ belongs to $BMO(\mathbb{R})$ and its Fourier transform \(\hat{f}\) is supported in $[0, \infty)$ and, $P_{y}$ is the Poisson Kernel, (see \cite{BoKy}).

Similarly, \(\mathrm{BMOA}(\mathbb{C}_{+})\) can be identified with the dual space of the  space \(H^{1}(\mathbb{C}_{+})\). In this article, we endow both \(\mathrm{BMO}(\mathbb{R})\) and \(\mathrm{BMOA}(\mathbb{C}_{+})\) with the dual norm.

%Since the Fourier transform of $f$ is supported in %$[0,\infty)$, it follows that, for every $z = x+iy \in %\mathbb{C}_{+}$,
%\[
%F(x+iy) = (P_y * f)(x) = \mathcal{P}_S(f)(z),
%\]
%where $\mathcal{P}_S$ denotes the Szeg\"o projection. In this case, we define the $BMOA$ norm of $F$ by
%\[
%\|F\|_{BMOA} := \inf \left\{ \|f\|_{BMO} : F = \mathcal{P}_S(f),\; f \in BMO(\mathbb{R}) \right\}.
%\]
\medskip

For any $s \in \mathbb{N}$, recall that $\mathcal{P}_{s}(\mathbb{R})$ denotes the polynomials with order not bigger than $s$. Assume that $f$ is a
 locally integrable function on $\mathbb{R}$. For any interval $I \subset  \mathbb{R}$ and  $s \in \mathbb{N}$, let $\mathcal{P}_{I}^{s}(f)$ be the
 minimizing polynomial $P \in \mathcal{P}_{s}(\mathbb{R})$ on $I$ such that, for all $Q \in \mathcal{P}_{s}(\mathbb{R})$, 
$$ \int_{I} (f(x) -P(x))Q(x)dx=0.  $$

Let  $s \in \mathbb{N}$ and  $\Phi$ a growth function of lower type. A locally integrable function $f$ on $\mathbb{R}$ is said to belong to the Campanato-Orlicz space   $\mathcal{L}^{\Phi,s}(\mathbb{R})$  if
$$  \|f\|_{\mathcal{L}^{\Phi,s}(\mathbb{R})}:= \sup_{I \subset \mathbb{R}} \left(\Phi^{-1}\left(\frac{1}{|I|}\right)\int_{I}|f(x)-\mathcal{P}_{I}^{s}(f)(x)|dx \right) < \infty,   $$
  where  the supremum is taking on all intervals $I \subset \mathbb{R}$.  

An analytic function $F$ on $\mathbb{C}_{+}$ is said to belong to the Campanato Hardy--Orlicz space on the upper half-plane, denoted by $\mathcal{H}\mathcal{L}^{\Phi,s}(\mathbb{C}_{+})$, if there exists a function $f \in \mathcal{L}^{\Phi,s}(\mathbb{R})$ whose Fourier transform is supported in $[0,\infty)$ such that
\[
F(x+iy) = (P_y * f)(x), \quad \forall\, x+iy \in \mathbb{C}_{+}.
\]
Equivalently, $\mathcal{H}\mathcal{L}^{\Phi,s}(\mathbb{C}_{+})$ can be defined as
\[
\mathcal{H}\mathcal{L}^{\Phi,s}(\mathbb{C}_{+}) := \mathcal{P}_S\big(\mathcal{L}^{\Phi,s}(\mathbb{R})\big),
\]
that is, the space of analytic functions $F$ on $\mathbb{C}_{+}$ that can be written in the form
\[
F = \mathcal{P}_S(f),
\]
for some $f \in \mathcal{L}^{\Phi,s}(\mathbb{R})$. This space is endowed with the quasi-norm
\[
\|F\|_{\mathcal{H}\mathcal{L}^{\Phi,s}} := \inf \left\{ \|f\|_{\mathcal{L}^{\Phi,s}} : F = \mathcal{P}_S(f),\; f \in \mathcal{L}^{\Phi,s}(\mathbb{R}) \right\}.
\]
If  $\Phi(t) = t$ and $\mathcal{P}_{I}^{s}(f)
   =m_{I}(f)$ then   $\mathcal{L}^{\Phi,s}(\mathbb{R})$ and   $\mathcal{H}\mathcal{L}^{\Phi,s}(\mathbb{C}_{+})$ coincide respectively with  $BMO(\mathbb{R})$ and  $BMOA(\mathbb{C}_{+})$ spaces. 

We have the following duality result.
\begin{theorem}\label{pro:main2aaop}
Let $\Phi \in \mathscr{L}$ and let $s \in \mathbb{N}$ be such that $s \geq m_{\Phi}$. Then the topological dual space of $H^{\Phi}(\mathbb{C}_{+})$, denoted by $\big(H^{\Phi}(\mathbb{C}_{+})\big)^{*}$, can be identified with the Campanato Hardy--Orlicz space $\mathcal{H}\mathcal{L}^{\Phi,s}(\mathbb{C}_{+})$. More precisely, for every $T \in \big(H^{\Phi}(\mathbb{C}_{+})\big)^{*}$, there exists a unique function $G \in \mathcal{H}\mathcal{L}^{\Phi,s}(\mathbb{C}_{+})$ such that, for all $F \in H^{\Phi}(\mathbb{C}_{+})$,
\[
T(F) := \lim_{y \to 0} \int_{\mathbb{R}} F(x+iy)\,\overline{G(x+iy)}\, dx.
\]

Moreover, the following norm equivalence holds:
\[
\|G\|_{\mathcal{H}\mathcal{L}^{\Phi,s}} \approx \sup \left\{ |T(F)| : F \in H^{\Phi}(\mathbb{C}_{+}),\ \|F\|_{H^{\Phi}(\mathbb{C}_{+})}^{\mathrm{lux}} \leq 1 \right\}.
\]
\end{theorem}

Let $b \in H^{2}(\mathbb{C}_{+})$. The Hankel operator with symbol $b$ is defined, for any bounded holomorphic function $g$ on $\mathbb{C}_{+}$, by
\[
h_{b}(g) := \mathcal{P}_{S}(b\, \overline{g}),
\]
where $\mathcal{P}_{S}$ denotes the Szeg\"o projection.

Let $\Phi_{1}$ and $\Phi_{2}$ be two growth functions of lower type, and let $b$ be a holomorphic function on $\mathbb{C}_{+}$. We say that the Hankel operator $h_{b}$ is bounded from $H^{\Phi_{1}}(\mathbb{C}_{+})$ to $H^{\Phi_{2}}(\mathbb{C}_{+})$ if there exists a constant $C>0$ such that
\begin{equation}\label{eq:ineitedaaqpqqehay}
\|h_{b}(g)\|_{H^{\Phi_{2}}(\mathbb{C}_{+})}^{\mathrm{lux}} \leq C \|g\|_{H^{\Phi_{1}}(\mathbb{C}_{+})}^{\mathrm{lux}}, \quad \forall\, g \in H^{\Phi_{1}}(\mathbb{C}_{+}).
\end{equation}
The operator norm of $h_b$ is defined by
\[
\|h_{b}\|:= \sup \left\{ \|h_{b}(g)\|_{H^{\Phi_{2}}(\mathbb{C}_{+})}^{\mathrm{lux}} : \|g\|_{H^{\Phi_{1}}(\mathbb{C}_{+})}^{\mathrm{lux}} \leq 1 \right\}.
\]

For $j \in \{1,2\}$, let $\Phi_{j}$ be a growth function of lower type $p_{j}$ and upper type $q_{j}$. 
Assume that either $0 < p_{2} \leq q_{2} < p_{1} < \infty$ or $0 < p_{1} \leq q_{1} < p_{2} < \infty$. 
If the operator $h_{b}: H^{\Phi_{1}}(\mathbb{C}_{+}) \to H^{\Phi_{2}}(\mathbb{C}_{+})$ is bounded, we say that $h_{b}$ satisfies a \emph{loss estimate} in the first case and a \emph{gain estimate} in the second case. In these situations, one obtains corresponding loss or gain estimates for the norm of $h_b$.

In the particular case where $\Phi_{j}(t) = t^{p_{j}}$, these definitions coincide with the classical notions of loss ($p_{2} < p_{1}$) and gain ($p_{1} < p_{2}$) estimates for Hankel operators.
Moreover, in the unit disc setting, these definitions are consistent with those introduced in \cite{djefeuto1}, namely that $H^{\Phi_{2}}(\mathbb{D}) \subset H^{\Phi_{1}}(\mathbb{D})$ in the loss case and $H^{\Phi_{1}}(\mathbb{D}) \subset H^{\Phi_{2}}(\mathbb{D})$ in the gain case. We have the following result.

\begin{theorem}\label{pro:mainfqmaqppaaqq5}
Let $\Phi_{1}, \Phi_{2}, \Phi_{3}$ be growth functions of lower type $p_{1}, p_{2}, p_{3}$, respectively. Assume that $\Phi_{2}$ is also of upper type $q_{2}$, with
$1 < p_{2} \leq q_{2} < \infty$. Let $b$ be a holomorphic function on $\mathbb{C}_{+}$. Then the following statements hold:
\begin{itemize}
\item[(i)] Assume that
$1 < p_{2} \leq q_{2} < p_{1} < \infty
\quad \text{and} \quad
\Phi_{3}^{-1}(t) \sim \frac{\Phi_{2}^{-1}(t)}{\Phi_{1}^{-1}(t)}$. Then the Hankel operator $h_{b} : H^{\Phi_{1}}(\mathbb{C}_{+}) \to H^{\Phi_{2}}(\mathbb{C}_{+})$ is bounded if and only if $b \in H^{\Phi_{3}}(\mathbb{C}_{+})$. Moreover,
\[
\|h_{b}\| \approx \|b\|_{H^{\Phi_{3}}}^{\mathrm{lux}}.
\]
\item[(ii)] If $\Phi\equiv\Phi_{1} \equiv \Phi_{2}$, then
$h_{b} : H^{\Phi}(\mathbb{C}_{+}) \to H^{\Phi}(\mathbb{C}_{+})$
is bounded if and only if $b \in BMOA(\mathbb{C}_{+})$. Moreover,
\[
\|h_{b}\| \approx \|b\|_{BMOA}.
\]

\item[(iii)] Assume that $0 < p_{1} \leq q_{1} \leq 1 < p_{2}
\quad \text{and} \quad
\Phi_{3}^{-1}(t) \sim \frac{t\,\Phi_{1}^{-1}(t)}{\Phi_{2}^{-1}(t)}$. Then the Hankel operator $h_{b} : H^{\Phi_{1}}(\mathbb{C}_{+}) \to H^{\Phi_{2}}(\mathbb{C}_{+})$ is bounded if and only if $b \in \mathcal{H}\mathcal{L}^{\Phi_{3},s}(\mathbb{C}_{+})$, for some $s \in \mathbb{N}$ such that $s \geq m_{\Phi_{3}}$. Moreover,
\[
\|h_{b}\| \approx \|b\|_{\mathcal{H}\mathcal{L}^{\Phi_{3},s}}.
\]
\end{itemize}
\end{theorem}

We use the  abbreviation $\mathrm{ A}\lesssim \mathrm{ B}$ for inequalities $\mathrm{ A}\leq C\mathrm{ B}$, where $C$ is a positive constant independent of the main parameters. If $\mathrm{ A}\lesssim \mathrm{ B}$ and $\mathrm{ B}\lesssim \mathrm{ A}$, then we write $\mathrm{ A} \approx \mathrm{ B}$.
In all what follows, the letter $C$ will be used for non-negative constants independent of the relevant variables that may change from one occurrence to another. Constants with subscript, such as $C_{s}$, may also change in different occurrences, but depend on the parameters mentioned in it.

\section{Preliminaries.}

We present in this section some useful results needed in our presentation.

\subsection{Growth functions.}

We  recall here some fundamental properties of growth functions needed in our discussion.

\medskip

Let  $\Phi$ be a growth function. We say that  $\Phi$ satisfies the $\Delta_{2}-$condition (or $\Phi \in \Delta_{2}$) if there exists a constant $K > 1$ such that
\begin{equation}\label{eq:delta2}
\Phi(2t) \leq K \Phi(t),~ \forall~ t >  0.\end{equation}
We say also that  $\Phi$ satisfies the $\nabla_{2}-$condition (or $\Phi \in \nabla_{2}$) if there exists $C > 1$ such that 
\begin{equation}\label{eq:delmta2}
\Phi(t) \leq \frac{1}{2C} \Phi(Ct),~ \forall~ t >  0.\end{equation}

Let  $\Phi$ be a convex growth function such that  $\lim_{t \to 0}\frac{\Phi(t)}{t}=0$ and $\lim_{t \to \infty}\frac{\Phi(t)}{t}=+\infty$. There exists  $\varphi$ a positive, left continuous and non-decreasing function on $[0, \infty)$ such that $\varphi(0)=0$ and $\lim_{t \to \infty}\varphi(t)=\infty$ and, 
$$   \Phi(t)=\int_{0}^{t}\varphi(s)ds, ~~\forall~t\geq 0.
  $$
The complementary function of $\Phi$ is the function $\Psi$ defined by
$$ \Psi(s)=\sup_{t\geq 0}\{st-\Phi(t) \}, ~ \forall~  s \geq 0.       $$
Note that function  $\Psi$ has the same properties as function $\Phi$ and, the complementary of $\Psi$ is $\Phi$ (see \cite{raoren}). Put
$$  a_\Phi:=\liminf_{t\to \infty}\frac{t\varphi(t)}{\Phi(t)}
   \hspace*{1cm}\textrm{and} \hspace*{1cm} b_\Phi:=\limsup_{t\to \infty}\frac{t\varphi(t)}{\Phi(t)}, $$
 and similarly $a_\Psi$, $b_\Psi$ be defined.

\medskip

Most of the results presented below are taken from reference \cite{djefeuto1}, which we will list in the sequel.

\begin{lemme}\label{pro:maiaqaq1q8}
Let  $\Phi$ be a growth function of lower type $p\in (0, \infty)$. The following assertions are satisfied: 
\begin{itemize}
\item[(i)]  $\Phi$ is equivalent to a continuous and increasing  growth function of lower type $p$.
\item[(ii)] The growth function $\Phi_{p}$ defined by
\begin{equation}\label{eq:suiaqaq8n}
\Phi_{p}(t)=\Phi\left( t^{1/p}\right),~~\forall~t\geq 0\end{equation}
is equivalent to a continuous, increasing and convex growth function.
\item[(iii)]   If $p \geq 1$ then $\Phi$ is equivalent to a continuous, increasing and convex growth function.
\item[(iv)] If $\Phi$ is also of upper type $q$.  Then $\Phi$ is equivalent to a continuous, increasing, and convex growth function belonging to $\mathscr{U} \cap \nabla_{2}$ (resp. $\mathscr{L}$) if and only if $1<p \leq q< \infty$ (resp.  $0<p \leq q \leq 1$).
\item[(v)]   $\Phi \in \mathscr{L}_{p}$ if and only if  $\Phi^{-1} \in \mathscr{U}^{1/p}$.

\end{itemize}
\end{lemme}

\begin{lemme}[Lemma 9, \cite{djefeuto1}]\label{pro:mainfqmpq5}
Let  $\Phi_{1}$ and $\Phi_{2}$  be two growth functions of lower type  $p_{1}$ and $p_{2}$ respectively. Let  $\Phi_{3}$ be a positive function on $[0, \infty)$ such that $\Phi_{3}^{-1} = \Phi_{1}^{-1}.\Phi_{2}^{-1},$ where  $\Phi_{j}^{-1}$ is the inverse function of  $\Phi_{j}$, for  $j\in \{1,2,3\}$. 
Then the function $\Phi_{3}$ is a growth function of lower type $r:=\left( 1/p_{1}+1/p_{2}  \right)^{-1}$. 
If, moreover, one of the functions  $\Phi_{1}$ or  $\Phi_{2}$ is of upper type $q$ then $\Phi_{3}$ is also of upper type $q$. In this case we have $0<r\leq q <\infty$.
\end{lemme}

\begin{lemme}[Lemma 10, \cite{djefeuto1}]\label{pro:mainfqmaqpq5}
Let  $\Phi_{1}$ be a growth function of lower type $p_{1}$ and  $\Phi_{2} \in \Delta_{2} \cap \nabla_{2}$  a convex  growth function. Let $\Phi_{3}$ be a  positive function   on $[0, \infty)$ such that
\begin{equation}\label{eq:deltAQa2}
\Phi_{3}^{-1}(t) =\Phi_{1}^{-1}(t)\Psi_{2}^{-1}(t),~ ~\forall~ t >  0,\end{equation}
where   $\Psi_{2}$  the complementary function of $\Phi_{2}$. The following assertions are satisfied:
\begin{itemize}
\item[(i)] If  $ b_{\Phi_{2}} < p_{1}$ then  $\Phi_{3}$ is a growth function of both lower type $p_{3}$ and upper type  $q_{3}$ such that  $1< p_{3} \leq q_{3}< \infty$.
\item[(ii)] If  $\Phi_{1}$ is also of upper type $q_{1}$ and  $0<p_{1} \leq q_{1} \leq  a_{\Phi_{2}}$ then   $\Phi_{3}$ is  a growth function of both lower type $p_{3}$ and upper type $q_{3}$ such that  $0<p_{3} \leq q_{3} \leq 1$.
\end{itemize} 
\end{lemme}

\medskip

\noindent
We also recall that equivalent growth functions generate the same Orlicz-type spaces: if $\Phi_1 \sim \Phi_2$, then
\[
L^{\Phi_1} = L^{\Phi_2}, \qquad H^{\Phi_1} = H^{\Phi_2},
\]
and their Luxemburg norms are equivalent. We can therefore, without loss of generality, replace the equivalence condition between  $\Phi_{1}$ and  $\Phi_{2}$ with an equality, (i.e:  $\Phi_{1}=\Phi_{2}$), when we work in Orlicz type spaces.

\medskip

\begin{remark}\label{pro:main 5aqaqq2pl}
\begin{itemize}
\item[(i)] From now on, every growth function $\Phi$ of lower type $p \in (0,\infty)$ will be assumed to be continuous and increasing. Moreover, by Lemma \ref{pro:maiaqaq1q8}, the function $\Phi_p$ defined in \eqref{eq:suiaqaq8n} is continuous, convex, and increasing.
\medskip
\item[(ii)] If $\Phi$ is simultaneously of lower type $p$ and upper type $q$, we shall assume without loss of generality that
\[
\Phi \in \mathscr{U} \cap \nabla_2 \quad \text{if } 1 < p \leq q < \infty,
\qquad \text{and} \qquad
\Phi \in \mathscr{L} \quad \text{if } 0 < p \leq q \leq 1,
\]
in view of Lemma \ref{pro:maiaqaq1q8}.
\end{itemize}
\end{remark}

\subsection{Hardy-Orlicz spaces.}

Let $s \in \mathbb{N}$ and let $\Phi$ be a growth function of lower type. The atomic Hardy--Orlicz space of order $s$ on $\mathbb{R}$, denoted by $H^{\Phi,s}_{\mathrm{at}}(\mathbb{R})$, is defined as the set of all tempered distributions $f \in \mathscr{S}'(\mathbb{R})$ that admit a decomposition
\[
f=\sum_{j=0}^{\infty} a_j \quad \text{in } \mathscr{S}'(\mathbb{R}),
\]
where, for each $j$, $a_j$ is a multiple of a $(\Phi,s)$-atom supported on an interval $I_j$, and such that
\[
\sum_{j=0}^{\infty} |I_j|\, \Phi\!\left( \|a_j\|_{L^2}\, |I_j|^{-1/2} \right) < \infty.
\]

For any sequence $\{a_j\}_j$ of multiples of $(\Phi,s)$-atoms, define
\[
\Lambda_2\big(\{a_j\}_j\big)
:= \inf \left\{ \lambda > 0 : 
\sum_{j=0}^{\infty} |I_j|\, 
\Phi\!\left( \frac{\|a_j\|_{L^2}\, |I_j|^{-1/2}}{\lambda} \right) \leq 1 
\right\}.
\]
The (Luxemburg) quasi-norm on $H^{\Phi,s}_{\mathrm{at}}(\mathbb{R})$ is then defined by
\[
\|f\|_{H^{\Phi,s}_{\mathrm{at}}(\mathbb{R})}^{\mathrm{lux}}
:= \inf \left\{ \Lambda_2\big(\{a_j\}_j\big) : 
f=\sum_{j=0}^{\infty} a_j \ \text{in } \mathscr{S}'(\mathbb{R}) 
\right\},
\]
where the infimum is taken over all such atomic decompositions of $f$.

\medskip

Let $f$ be a bounded tempered distribution on $\mathbb{R}$. The Poisson maximal function of $f$, denoted by $f_P^*$, is defined by
\[
f_P^*(x) := \sup_{t>0} |(P_t * f)(x)|, \quad \forall x \in \mathbb{R},
\]
where, for each $t>0$, $P_t$ is the Poisson kernel on $\mathbb{R}$ given by
\[
P_t(x) := \frac{1}{\pi} \frac{t}{x^2 + t^2}, \quad \forall x \in \mathbb{R}.
\]

It is well known that, if $f$ is a bounded tempered distribution on $\mathbb{R}$, then $P_t * f$ is a well-defined, bounded, and smooth function, harmonic on the upper half-plane $\mathbb{C}_+$, and
\[
P_t * f \longrightarrow f \quad \text{in } \mathscr{S}'(\mathbb{R}) \quad \text{as } t \to 0.
\]

We recall some fundamental results on the Hardy--Orlicz space $H^{\Phi}(\mathbb{R})$, most of which can be found in \cite{yangLiangKy}.

\begin{proposition}\label{pro:mainamp}
Let $\Phi \in \mathscr{L}$ and let $s \in \mathbb{N}$ be such that $s \geq m_{\Phi}$, where $m_{\Phi}$ is defined in \eqref{eq:dqtap2}. Then the following assertions hold:
\begin{itemize}
\item[(i)] Let $f$ be a bounded tempered distribution on $\mathbb{R}$. Then $f \in H^{\Phi}(\mathbb{R})$ if and only if $f_P^* \in L^{\Phi}(\mathbb{R})$. Moreover, the corresponding quasi-norms are equivalent.

\item[(ii)] The inclusion $H^{\Phi}(\mathbb{R}) \subset \mathscr{S}'(\mathbb{R})$ is continuous.
\item[(iii)] If $s \geq m_{\Phi}$, then $H^{\Phi}(\mathbb{R}) = H^{\Phi,s}_{\mathrm{at}}(\mathbb{R})$
with equivalent quasi-norms.
\end{itemize}
\end{proposition}

\medskip

\noindent
Proposition~\ref{pro:mainamp} shows that the Hardy--Orlicz space $H^{\Phi}(\mathbb{R})$ admits both a maximal characterization and an atomic decomposition. These equivalent descriptions will be used interchangeably in the sequel, depending on whether localization or harmonic analysis tools are required.

\medskip

\noindent
We now prove a useful estimate that controls the action of distributions in $H^{\Phi}(\mathbb{R})$ on Schwartz functions. 

\begin{proposition}\label{pro:maaqqaq5}
Let  $\Phi \in \mathscr{L}$ and  $f \in H^{\Phi}(\mathbb{R})$. For all  $\varphi \in \mathscr{S}(\mathbb{R})$, we have
\begin{equation}\label{eq:tqa2}
|\langle f, \varphi\rangle| \leq \mathcal{N}_{m_{\Phi}}(\varphi)\inf_{z \in \mathbb{R}:|z| \leq 1}f_{m_{\Phi}}^{*}(z). 
  \end{equation}
Moreover,  there is a constant $C:=C_{\Phi}>0$ such that
\begin{equation}\label{eq:tqqaa2}
|\langle f, \varphi\rangle| \leq C \mathcal{N}_{m_{\Phi}}(\varphi)\|f\|_{H^{\Phi}(\mathbb{R})}^{lux}. 
  \end{equation}
\end{proposition}

\begin{proof} 
Assume $f\not=0$ because when $f=0$ there is nothing to show.
Let  $\varphi \in \mathscr{S}(\mathbb{R})$. For $x \in \mathbb{R}$, put 
$\widetilde{\varphi}(x) =\varphi(-x).$
By construction,  $\widetilde{\varphi} \in \mathscr{S}(\mathbb{R})$,  $\mathcal{N}_{m_{\Phi}}(\widetilde{\varphi})=\mathcal{N}_{m_{\Phi}}(\varphi)$. We have
$$ |\langle f, \varphi\rangle|= \left|\int_{\mathbb{R}}f(x)\varphi(x)dx \right|=\left|\int_{\mathbb{R}}f(x)\widetilde{\varphi}(0-x)dx \right| =|(f\ast \widetilde{\varphi})(0)|.  $$
Let $z \in \mathbb{R}$ such that $|z| < 1$. We have
$$ |(f\ast \widetilde{\varphi})(0)| \leq \sup_{|z|<1}|(f\ast \widetilde{\varphi})(z)| \leq \mathcal{N}_{m_{\Phi}}(\widetilde{\varphi}) \sup_{|z|<1}\left|\left(f\ast \frac{\widetilde{\varphi}}{\mathcal{N}_{m_{\Phi}}(\widetilde{\varphi})}\right)(z)\right| \leq \mathcal{N}_{m_{\Phi}}(\widetilde{\varphi})f_{m_{\Phi}}^{*}(z). $$
We deduce that, 
$$  |\langle f, \varphi\rangle| \leq \mathcal{N}_{m_{\Phi}}(\varphi)\inf_{z \in \mathbb{R}:|z| < 1}f_{m_{\Phi}}^{*}(z). 
   $$
It follows that
\begin{align*}
\Phi\left(  \frac{|\langle f, \varphi\rangle|}{\mathcal{N}_{m_{\Phi}}(\varphi)\|f\|_{H^{\Phi}(\mathbb{R})}^{lux}}  \right) &\leq \frac{1}{|I(0,2)|}\int_{I(0,2)}\Phi\left(  \frac{\inf_{z \in \mathbb{R}:|z| < 1}f_{m_{\Phi}}^{*}(z)}{\|f\|_{H^{\Phi}(\mathbb{R})}^{lux}}  \right)dx  \\
 &\leq \frac{1}{2}\int_{\mathbb{R}}\Phi\left(  \frac{f_{m_{\Phi}}^{*}(x)}{\|f\|_{H^{\Phi}(\mathbb{R})}^{lux}}  \right)dx \leq \frac{1}{2},
\end{align*}
where $I(0,2)$ is the interval with center $0$ and length $2$. 
\end{proof}

\begin{proposition}\label{pro:main2aaoamp}
Let $\Phi \in \mathscr{L}$.   Then
\[
\mathscr{S}(\mathbb{R}) \subset H^{\Phi}(\mathbb{R}) \subset \mathscr{S}'(\mathbb{R}).
\]
\end{proposition}

\begin{proof}
The inclusion $H^{\Phi}(\mathbb{R}) \subset \mathscr{S}'(\mathbb{R})$ follows from Proposition~\ref{pro:mainamp}. It remains to prove that $\mathscr{S}(\mathbb{R}) \subset H^{\Phi}(\mathbb{R})$.

Let $\varphi \in \mathscr{S}(\mathbb{R})$. By the maximal characterization of $H^{\Phi}(\mathbb{R})$, it suffices to show that $\varphi_P^* \in L^{\Phi}(\mathbb{R})$, where
\[
\varphi_P^*(x) := \sup_{t>0} |(P_t * \varphi)(x)|.
\]
Since $\varphi \in \mathscr{S}(\mathbb{R})$, for every $N \in \mathbb{N}$,
\[
|\varphi(y)| \le C_N (1+|y|)^{-N}.
\]
A standard estimate yields
\[
|(P_t * \varphi)(x)| \le C_N (1+|x|)^{-N},
\]
uniformly in $t>0$. Hence,
\[
\varphi_P^*(x) \le C_N (1+|x|)^{-N}.
\]
Since $\Phi$ is of lower type $p \in (0,1]$, we obtain
\[
\Phi(\varphi_P^*(x)) \le C (1+|x|)^{-Np}.
\]
Choosing $N$ large enough gives
\[
\int_{\mathbb{R}} \Phi(\varphi_P^*(x))\,dx < \infty,
\]
which proves that $\varphi \in H^{\Phi}(\mathbb{R})$.
\end{proof}
\medskip

\noindent
In the following result, we obtain that $H^{\Phi}(\mathbb{R})$ closed under regularization by the Poisson kernel, this provides an approximation of the identity compatible with the Hardy--Orlicz structure.

\begin{proposition}\label{pro:mainqaq2aop}
Let  $\Phi \in \mathscr{L}$ and $f \in \mathscr{S}'(\mathbb{R})$ be a bounded distribution. 
If $f \in H^{\Phi}(\mathbb{R})$, then for $t>0$, the function $P_{t}\ast f$ belongs to $H^{\Phi}(\mathbb{R})$ and 
\begin{equation}\label{eq:pmta2}
\|P_{t}\ast f\|_{H^{\Phi}(\mathbb{R})}^{lux} \leq C \|f\|_{H^{\Phi}(\mathbb{R})}^{lux},
\end{equation}
where $C$ is a constant independent of $f$ and $t$. Moreover,
\begin{equation}\label{eq:dqtaqa2}
\lim_{t \to 0}\|(P_{t}\ast f)-f\|_{H^{\Phi}(\mathbb{R})}^{lux}=0.
\end{equation}
\end{proposition}

\begin{proof}
Since $f$ is a bounded tempered distribution on $\mathbb{R}$, for every $s>0$ the convolution $P_s*f$ is well defined, bounded and continuous on $\mathbb{R}$. In particular, $(P_t)_{t>0}$ is an approximation of the identity. Hence, for all $x\in\mathbb{R}$ and all $s>0$,
\[
\lim_{t\to 0} P_t*(P_s*f)(x)=P_s*f(x).
\]
Using the semigroup property $P_t*(P_s*f)=P_{s+t}*f$, we obtain
\begin{equation}\label{eq:pointwise_limit}
\lim_{t\to 0} P_{s+t}*f(x)=P_s*f(x),
\qquad \forall s>0,\ \forall x\in\mathbb{R}.
\end{equation}
For $t>0$ and $s>0$,
\[
|P_s*(P_t*f)(x)|
=
|P_{s+t}*f(x)|
\le
\sup_{u>0}|P_u*f(x)|
=
f_P^*(x).
\]
Thus
\[
(P_t*f)_P^*(x)\le f_P^*(x).
\]
By the definition of the Hardy–Orlicz norm,
\[
\|P_t*f\|_{H^\Phi(\mathbb{R})}^{lux}
\approx
\|(P_t*f)_P^*\|_{L^\Phi}^{lux}
\le
\|f_P^*\|_{L^\Phi}^{lux}
\approx
\|f\|_{H^\Phi(\mathbb{R})}^{lux}.
\]
Let
\[
u(s,x):=P_s*f(x), \qquad \forall~s>0,\ \forall~ x\in\mathbb{R}.
\]
It is well known that $u(s,\cdot)\longrightarrow f$
in  $\mathscr S'(\mathbb{R})$  as  $s\to 0$.
Hence, for every $\varepsilon>0$ and every $\varphi\in\mathscr S(\mathbb{R})$, there exists
$k_{\varepsilon}(\varphi)>0$ such that
\[
|\langle u(s,\cdot)-f,\varphi\rangle|
<
\varepsilon,
\qquad \forall~ 0<s<k_{\varepsilon}(\varphi).
\]
Let $x\in\mathbb{R}$.  
Choose $\psi\in\mathscr S(\mathbb{R})$ such that
$\int_{\mathbb R}\psi(y)\,dy=1$,
and define
\[
\psi_\eta(y)=\frac1\eta\,\psi\!\left(\frac{x-y}{\eta}\right),
\qquad \forall~ \eta>0 .
\]
Then $(\psi_\eta)_{\eta>0}$ is an approximation of the Dirac mass at $x$, that is $\psi_\eta\longrightarrow \delta_x$ in  $\mathscr S'(\mathbb{R})$  as  $\eta\to 0$.
Applying the previous estimate with $\varphi=\psi_\eta$, we obtain that for every
$\varepsilon>0$ there exists $k_{\varepsilon}(\psi_\eta)>0$ such that
\[
|\langle u(s,\cdot)-f,\psi_\eta\rangle|
<
\varepsilon,
\qquad \forall~ 0<s<k_{\varepsilon}(\psi_\eta).
\]
In particular, for all $0<s,t<k_{\varepsilon}(\psi_\eta)/2$,
\[
|\langle u(s+t,\cdot)-u(s,\cdot),\psi_\eta\rangle|
\le
|\langle u(s+t,\cdot)-f,\psi_\eta\rangle|
+
|\langle u(s,\cdot)-f,\psi_\eta\rangle|
<
2\varepsilon .
\]
Since
\[
\langle u(s+t,\cdot)-u(s,\cdot),\psi_\eta\rangle
=
\int_{\mathbb R}
\bigl(u(s+t,y)-u(s,y)\bigr)\psi_\eta(y)\,dy,
\]
and since $u(s,\cdot)$ is continuous, the family $(\psi_\eta)$ converges to $\delta_x$,
so that
\[
\lim_{\eta\to 0}
\langle u(s+t,\cdot)-u(s,\cdot),\psi_\eta\rangle
=
u(s+t,x)-u(s,x).
\]
Consequently, letting $\eta\to 0$ in the previous inequality yields
\[
|u(s+t,x)-u(s,x)|
\le
2\varepsilon,
\qquad \forall~ 0<s,t<k_{\varepsilon}(x)/2,
\]
for almost every $x\in\mathbb R$. Therefore
\begin{equation}\label{eq:small_s}
\sup_{0<s<k_\varepsilon(x)/2}
|u(s+t,x)-u(s,x)|
\le
2\varepsilon,
\qquad \forall~ 0<t<k_\varepsilon(x)/2 .
\end{equation}
On the other hand, since the Poisson kernel is smooth for $s>0$, the mapping
$s\mapsto u(s,x)$ is $\mathscr C^1$ and
\[
\partial_s u(s,x)
=
(\partial_s P_s)*f(x).
\]
It is well known that
\[
|\partial_s P_s(y)|
\lesssim
\frac1s P_s(y).
\]
Hence
\[
|\partial_s u(s,x)|
\lesssim
\frac1s f_P^*(x).
\]
Using the fundamental theorem of calculus,
\[
u(s+t,x)-u(s,x)
=
\int_s^{s+t}\partial_r u(r,x)\,dr.
\]
Therefore
\[
|u(s+t,x)-u(s,x)|
\lesssim
f_P^*(x)\int_s^{s+t}\frac{dr}{r}
\lesssim
\frac{t}{s}f_P^*(x).
\]
Thus
\begin{equation}\label{eq:large_s}
\sup_{s\ge k_\varepsilon(x)/2}
|u(s+t,x)-u(s,x)|
\lesssim
\frac{t}{k_\varepsilon(x)}f_P^*(x).
\end{equation}
Combining \eqref{eq:small_s} and \eqref{eq:large_s}, we obtain
\[
(P_t*f-f)_P^*(x)
=
\sup_{s>0}|u(s+t,x)-u(s,x)|
\lesssim
2\varepsilon
+
\frac{t}{k_\varepsilon(x)}f_P^*(x), 
\qquad \forall~ 0<s,t<k_{\varepsilon}(x)/2,
\]
for almost every $x\in\mathbb R$. Letting $t\to 0$ gives
\[
(P_t*f-f)_P^*(x)\longrightarrow0
\quad \text{for almost every } x\in\mathbb R .
\]
Moreover,
\[
(P_t*f-f)_P^*(x)
\le
2f_P^*(x).
\]
Since $\Phi\in\mathscr L$, there exists $K>0$ such that
\[
\Phi(2u)\le K\Phi(u).
\]
Thus
\[
\Phi((P_t*f-f)_P^*(x))
\le
K\Phi(f_P^*(x)).
\]
Because $f\in H^\Phi(\mathbb R)$, we have
$\Phi(f_P^*)\in L^1(\mathbb R)$.
Hence, by the dominated convergence theorem,
\[
\int_{\mathbb R}
\Phi\!\left(\varepsilon (P_t*f-f)_P^*\right)dx
\longrightarrow0
\quad (t\to0).
\]
By the definition of the Luxemburg norm,
\[
\|P_t*f-f\|_{H^\Phi(\mathbb R)}^{lux}
\longrightarrow0 .
\]

This completes the proof.
\end{proof}

\medskip

\noindent

Let $\alpha>0$ and $F$ be a harmonic function on $\mathbb{C_{+}}$. The  non-tangential  maximal  function  $F_{\alpha}^{*}$ of  $F$ is defined by 
 $$   F_{\alpha}^{*}(t):=\sup_{z\in \Gamma_{\alpha}(t)}|F(z)|, ~~\forall~ t \in \mathbb{R},  $$
 where  $ \Gamma_{\alpha}(t) := \left\{ x+iy \in \mathbb{C_{+}}: | x-t | < \alpha y  \right\}$. 

\begin{lemme}\label{pro:main6aapaqmaqq4k}
Let  $\Phi \in \mathscr{L} \cup \mathscr{U}$ and   $F$  an analytic function on $\mathbb{C}_{+}$.  The following assertions are equivalent:
\begin{itemize}
  \item[(i)] $F \in H^{\Phi}(\mathbb{C}_{+})$;
  \item[(ii)]  $F_{\alpha}^{*} \in L^{\Phi}(\mathbb{R})$,  for all  $\alpha>0$.
 \end{itemize}
 Moreover, $\|F_{\alpha}^{*}\|_{L^{\Phi}(\mathbb{R})}^{lux} \approx \|F\|_{H^{\Phi}(\mathbb{C}_{+})}^{lux}$.
 \end{lemme}

The proof of Lemma \ref{pro:main6aapaqmaqq4k} is identical to that of \cite[Lemma 15]{djefeuto1}. Therefore, it will be omitted.

\medskip

\noindent
We now analyze the behavior of $(\Phi,s)$-atoms under the Szeg\"o projection, which is a key step toward proving boundedness results for general elements of $H^{\Phi}(\mathbb{R})$.

\begin{lemme}\label{pro:main2aaqop}
Let  $\Phi\in   \mathscr{L}$ and $s \in \mathbb{N}$ satisfies $s \geq m_{\Phi}$. If $a$ is   $(\Phi,s)-$atom on $\mathbb{R}$ then $\mathcal{P}_{S}(a)$ belongs to  $H^{\Phi}(\mathbb{C}_{+})$. Moreover, there exists a constant $C:=C_{\Phi}>0$ such that for all $\lambda>0$,
\begin{equation}\label{eq:deaqaqa2}
\sup_{y>0}\int_{\mathbb{R}}\Phi\left( \frac{|\mathcal{P}_{S}(a)(x+iy)|}{\lambda}\right)dx \leq C |I|\Phi\left( \frac{|I|^{-1/2}\| a\|_{L^{2}}}{\lambda}\right), 
\end{equation}
where $I$ is is an interval of finite length such that \text{supp} $(a)\subset I$.
\end{lemme}

\begin{proof} 
Without loss of generality, we fix $\lambda=1$.
Suppose that \text{supp} $(a)\subset I$, where $I:=I(x_{0})$ is the interval with center $x_{0} \in \mathbb{R}$. For $y>0$, we have
 $$  \int_{\mathbb{R}}\Phi( |\mathcal{P}_{S}(a)(x+iy)|)dx = \left(\int_{\widetilde{I}}+\int_{\mathbb{R}\backslash \widetilde{I}}\right)\Phi( |\mathcal{P}_{S}(a)(x+iy)|)dx,    $$
where $\widetilde{I}:=\widetilde{I}(x_{0})$ is an interval with center $x_{0}$ and  $|\widetilde{I}|= 2|I|$
Since  $\mathcal{P}_{S}$ is bounded from  $L^{2}(\mathbb{R})$ in  $H^{2}(\mathbb{C}_{+})$, we have
\begin{align*}
\int_{\widetilde{I}}\Phi( |\mathcal{P}_{S}(a)(x+iy)|)dx 
&\lesssim |\widetilde{I}|\Phi\left( \frac{1}{|\widetilde{I}|}\int_{\widetilde{I}} |\mathcal{P}_{S}(a)(x+iy)|dx \right) \\
&\lesssim |\widetilde{I}|\Phi\left( \frac{1}{|\widetilde{I}|}\left(\int_{\mathbb{R}} |\mathcal{P}_{S}(a)(x+iy)|^{2}dx\right)^{1/2}\left( \int_{\mathbb{R}}\chi_{\widetilde{I}}(x)^{2}dx\right)^{1/2} \right) \\
&\lesssim |\widetilde{I}|\Phi\left( \left(\int_{\mathbb{R}} |a(x)|^{2}dx\right)^{1/2}|\widetilde{I}|^{-1/2} \right), \end{align*}
thanks to   H\"older's inequality.  We deduce that
\begin{equation}\label{eq:dta2}
\int_{\widetilde{I}}\Phi( |\mathcal{P}_{S}(a)(x+iy)|)dx \lesssim |I|\Phi\left( \| a\|_{L^{2}}|I|^{-1/2}  \right).
\end{equation}
Let $z \in \mathbb{C}_{+}$. Consider $\varphi_{z}$ the function defined by
$$ \varphi_{z}(t) = \frac{1}{t-z},~~\forall~t\in \mathbb{R}.  $$
By construction, $\varphi_{z} \in \mathscr{C}^{\infty}(\mathbb{R})$ and 
$$ \varphi_{z}^{(n)}(t) = (-1)^{n}\frac{n!}{(t-z)^{n+1}},~~\forall~n \in \mathbb{N},~~\forall~t\in \mathbb{R}.  $$
Since $\varphi_{z}$ admits a Taylor series development around the point $x_{0}$, for all $t\in I$, we have
$$  \varphi_{z}(t)= \sum_{n=0}^{\infty}\frac{\varphi_{z}^{(n)}(x_{0})(t-x_{0})^{n}}{n!} 
= \sum_{n=0}^{\infty} \frac{(-1)^{n}}{(x_{0}-z)^{n+1}}(t-x_{0})^{n}.
   $$
It follows that
\begin{align*}
|\mathcal{P}_{S}(a)(z)| &= \left|\frac{1}{2\pi i} \int_{\mathbb{R}}\frac{a(t)}{t-z}dt    \right|  = \left| \frac{1}{2\pi i}\int_{I} \varphi_{z}(t)a(t)dt    \right| \\
&=\left|\left(\sum_{n=0}^{s}+ \sum_{n=s+1}^{\infty} \right) \int_{I}\frac{(-1)^{n}}{(x_{0}-z)^{n+1}}(t-x_{0})^{n} a(t)dt    \right| \\
&=\left|\sum_{n=0}^{s}\sum_{j=0}^{n}\frac{(-1)^{2n-j}C_{n}^{j}x_{0}^{n-j}}{(x_{0}-z)^{n+1}}\int_{I}a(t)t^{j}dt + \sum_{n=s+1}^{\infty}  \int_{I}\frac{(-1)^{n}}{(x_{0}-z)^{n+1}}(t-x_{0})^{n} a(t)dt    \right| \\
&=\left|\sum_{n=s+1}^{\infty}  \int_{I}\frac{(-1)^{n}}{(x_{0}-z)^{n+1}}(t-x_{0})^{n} a(t)dt    \right|, 
\end{align*}
since $a$ is   $(\Phi,s)-$atom on $\mathbb{R}$. We deduce that 
\begin{equation}\label{eq:dtka2}
 |\mathcal{P}_{S}(a)(z)| \leq  \sum_{n=s+1}^{\infty} \frac{|I|^{n+\frac{1}{2}}}{|x_{0}-z|^{n+1}}\| a\|_{L^{2}}, ~~\forall~ z\in \mathbb{C}_{+},
\end{equation}
since $x_{0}$ and $t$ belong to the interval $I$ and $\int_{I}|a(t)|dt \leq |I|^{1/2}\| a\|_{L^{2}}$.
For $k \geq 0$, put $$ E_{k}(x_{0}):= \{ x \in \mathbb{R}: 2^{k}|I| < |x-x_{0}| \leq 2^{k+1}|I|  \}.    $$
Let $z=x+iy\in \mathbb{C}_{+}$ such that  $x \in E_{k}(x_{0})$. We have
$$ |x_{0}-z| \geq |x_{0}-x| > 2^{k}|I|.   $$
It follows that
\begin{equation}\label{eq:dtkoa2}
 |\mathcal{P}_{S}(a)(x+iy)| \leq \sum_{n=s+1}^{\infty} \frac{1}{2^{k(n+1)}}|I|^{-1/2}\| a\|_{L^{2}}, ~~\forall~x \in E_{k}(x_{0}), ~~\forall~y>0.
\end{equation}
For $y>0$, we have: 
\begin{align*}
\int_{\mathbb{R}\backslash \widetilde{I}}\Phi( |\mathcal{P}_{S}(a)(x+iy)|)dx &= \sum_{k=0}^{\infty} \int_{E_{k}(x_{0})}\Phi( |\mathcal{P}_{S}(a)(x+iy)|)dx \\
&\lesssim \sum_{k=0}^{\infty} \Phi\left( \sum_{n=s+1}^{\infty} \frac{1}{2^{k(n+1)}}|I|^{-1/2}\| a\|_{L^{2}}\right)|E_{k}(x_{0})| \\
&\lesssim \sum_{k=0}^{\infty} \sum_{n=s+1}^{\infty}\Phi\left(  \frac{1}{2^{k(n+1)}}|I|^{-1/2}\| a\|_{L^{2}}\right)2^{k}|I| \\
&\lesssim \sum_{k=0}^{\infty} \sum_{n=s+1}^{\infty}\frac{1}{2^{k(n+1)a_{\Phi}}}\Phi\left( |I|^{-1/2}\| a\|_{L^{2}}\right)2^{k}|I| \\
&= |I|\Phi\left( |I|^{-1/2}\| a\|_{L^{2}}\right)\sum_{k=0}^{\infty} \sum_{n=s+1}^{\infty}\frac{1}{2^{k[(n+1)a_{\Phi}-1]}} \\
&\lesssim |I|\Phi\left( |I|^{-1/2}\| a\|_{L^{2}}\right), 
\end{align*}
since  $(n+1)a_{\Phi}-1 >0$. We deduce that 
\begin{equation}\label{eq:dtkoaqa2}
 \int_{\mathbb{R}\backslash \widetilde{I}}\Phi( |P_{S}(a)(x+iy)|)dx \lesssim |I|\Phi\left( |I|^{-1/2}\| a\|_{L^{2}}\right), ~~\forall~y>0.
 \end{equation}
\end{proof}

Let $f$ a distribution on $\mathbb{R}$. The Hilbert transform $\mathcal{H}(f)$ of $f$ is defined by
\[
\mathcal{H}(f)(x)
:=
\lim_{\varepsilon\to 0}
\frac{1}{\pi}
\int_{|x-t|>\varepsilon}
\frac{f(t)}{x-t}\,dt,
\qquad \forall x\in\mathbb{R},
\]
whenever the limit exists.

\begin{theorem}\label{pro:main5aq5}
Let $\Phi \in \mathscr{L}$ and let $f \in H^{\Phi}(\mathbb{R})$. Then the following assertions hold:
\begin{itemize}
\item[(i)] The Hilbert transform $\mathcal{H}(f)$ belongs to $H^{\Phi}(\mathbb{R})$ and
\begin{equation}\label{eq:deaqaa2}
\|\mathcal{H}(f)\|_{H^{\Phi}(\mathbb{R})}^{\mathrm{lux}}
\leq C_{1}\,\|f\|_{H^{\Phi}(\mathbb{R})}^{\mathrm{lux}},
\end{equation}
where $C_{1}$ is independent of $f$.

\item[(ii)] The Szeg\"o projection $\mathcal{P}_{S}(f)$ belongs to $H^{\Phi}(\mathbb{C}_{+})$. Moreover,
\begin{equation}\label{eq:deaqaaqqa2}
\|\mathcal{P}_{S}(f)\|_{H^{\Phi}(\mathbb{C}_{+})}^{\mathrm{lux}}
\leq C_{2}\,\|f\|_{H^{\Phi}(\mathbb{R})}^{\mathrm{lux}},
\end{equation}
where $C_{2}$ is a positive constant independent of $f$.
\end{itemize}
\end{theorem}

\begin{proof}
We only prove (i) and (ii) successively. We may assume without loss of generality that $f \neq 0$.

\medskip

\noindent\textbf{Proof of (i).}
Let $s \in \mathbb{N}$ be such that $s \ge m_{\Phi}$. By the atomic decomposition of $H^{\Phi}(\mathbb{R})$, there exists a sequence of $(\Phi,s)$-atoms $(a_j)_j$ such that
\[
f=\sum_{j} a_j \quad \text{in } \mathscr{S}'(\mathbb{R}),
\quad\text{and}\quad
\|f\|_{H^{\Phi}(\mathbb{R})}^{\mathrm{lux}}
\approx 
\|f\|_{H^{\Phi,s}_{\mathrm{at}}(\mathbb{R})}^{\mathrm{lux}}.
\]
Let $I_j$ be intervals such that $\mathrm{supp}(a_j)\subset I_j$.

Since $a_j \in L^2(\mathbb{R})$ and $\mathcal{H}$ is bounded on $L^2(\mathbb{R})$, we have $\mathcal{H}(a_j)\in L^2(\mathbb{R})$. Moreover, using the boundary behavior of the Szeg\"o projection, we have for almost every $x\in\mathbb{R}$,
\[
\lim_{y\to 0}\mathcal{P}_S(a_j)(x+iy)
=\frac12\big(a_j(x)+i\mathcal{H}(a_j)(x)\big).
\]
This implies the pointwise control
\[
(\mathcal{H}(a_j))_P^*(x)
\lesssim (\mathcal{P}_S(a_j))^*(x).
\]

By Lemma~\ref{pro:main2aaqop}, for all $y>0$,
\[
\int_{\mathbb{R}}
\Phi\!\left(
\frac{|\mathcal{P}_S(a_j)(x+iy)|}
{\|f\|_{H^{\Phi,s}_{\mathrm{at}}}^{\mathrm{lux}}}
\right)dx
\lesssim
|I_j|\,
\Phi\!\left(
\frac{|I_j|^{-1/2}\|a_j\|_{L^2}}
{\|f\|_{H^{\Phi,s}_{\mathrm{at}}}^{\mathrm{lux}}}
\right).
\]
Consequently,
\[
\int_{\mathbb{R}}
\Phi\!\left(
\frac{|(\mathcal{H}(a_j))_P^*(x)|}
{\|f\|_{H^{\Phi,s}_{\mathrm{at}}}^{\mathrm{lux}}}
\right)dx
\lesssim
|I_j|\,
\Phi\!\left(
\frac{|I_j|^{-1/2}\|a_j\|_{L^2}}
{\|f\|_{H^{\Phi,s}_{\mathrm{at}}}^{\mathrm{lux}}}
\right).
\]

Using the linearity of $\mathcal{H}$ and the subadditivity of $\Phi$, we obtain
\[
\int_{\mathbb{R}}
\Phi\!\left(
\frac{|(\mathcal{H}(f))_P^*(x)|}
{\|f\|_{H^{\Phi}(\mathbb{R})}^{\mathrm{lux}}
}
\right)dx
\lesssim
\sum_j
|I_j|\,
\Phi\!\left(
\frac{|I_j|^{-1/2}\|a_j\|_{L^2}}
{\|f\|_{H^{\Phi,s}_{\mathrm{at}}}^{\mathrm{lux}}}
\right).
\]
By the definition of the atomic Luxemburg norm, the right-hand side is bounded by a constant independent of $f$. Hence,
\[
\int_{\mathbb{R}}
\Phi\!\left(
\frac{|(\mathcal{H}(f))_P^*(x)|}
{\|f\|_{H^{\Phi}(\mathbb{R})}^{\mathrm{lux}}
}
\right)dx
\lesssim 1,
\]
which proves \eqref{eq:deaqaa2}.

\medskip

\noindent\textbf{Proof of (ii).}
Using again the atomic decomposition of $f$ and Lemma~\ref{pro:main2aaqop}, for every $y>0$,
\begin{align*}
\int_{\mathbb{R}}
\Phi\!\left(
\frac{|\mathcal{P}_S(f)(x+iy)|}
{\|f\|_{H^{\Phi}(\mathbb{R})}^{\mathrm{lux}}}
\right)dx
&\lesssim
\sum_j
\int_{\mathbb{R}}
\Phi\!\left(
\frac{|\mathcal{P}_S(a_j)(x+iy)|}
{\|f\|_{H^{\Phi,s}_{\mathrm{at}}}^{\mathrm{lux}}}
\right)dx \\
&\lesssim
\sum_j
|I_j|\,
\Phi\!\left(
\frac{|I_j|^{-1/2}\|a_j\|_{L^2}}
{\|f\|_{H^{\Phi,s}_{\mathrm{at}}}^{\mathrm{lux}}}
\right).
\end{align*}
By the definition of the Luxemburg norm, the right-hand side is bounded. Therefore,
\[
\sup_{y>0}
\int_{\mathbb{R}}
\Phi\!\left(
\frac{|\mathcal{P}_S(f)(x+iy)|}
{\|f\|_{H^{\Phi}(\mathbb{R})}^{\mathrm{lux}}}
\right)dx
\lesssim 1,
\]
which yields \eqref{eq:deaqaaqqa2} and completes the proof.
\end{proof}

\medskip

%\noindent
%Theorem~\ref{pro:main5aq5} highlights the stability of $H^{\Phi}(\mathbb{R})$ under singular integral operators and analytic projections, which are central tools in harmonic analysis.

The following is a well known fact.
\begin{proposition}[Proposition 3.16, \cite{djesehb1}]\label{pro:main6mq4}
Let  $\Phi \in \mathscr{L} \cup \mathscr{U}$. For  $F \in H^{\Phi}(\mathbb{C}_{+})$, we have
 \begin{equation}\label{eq:ualp5leson}
 |F(x+iy)|\leq \Phi^{-1}\left(\frac{2}{\pi y}\right)\|F\|_{H^{\Phi}(\mathbb{C}_{+})}^{lux}, ~~ \forall~x+iy \in \mathbb{C_{+}}.\end{equation}
 \end{proposition}

From the last two results, we obtain the following.
\begin{corollaire}\label{pro:main5aqpq5}
Let $\Phi \in \mathscr{L}$ and let $f$ be a bounded tempered distribution on $\mathbb{R}$. If $f \in H^{\Phi}(\mathbb{R})$, then for every $z \in \mathbb{C}_{+}$, we have
\begin{equation}\label{eq:uasoqan}
\left|\int_{\mathbb{R}}\frac{f(t)}{t-\overline{z}}\,dt\right|
\leq C\, \|f\|_{H^{\Phi}(\mathbb{R})}^{\mathrm{lux}},
\end{equation}
where $C>0$ is a constant depending only on $\Phi$ and $z$.
\end{corollaire}

\begin{proof}
Since $f$ is a bounded tempered distribution on $\mathbb{R}$, the function $(x,y)\mapsto (P_y*f)(x)$ is harmonic on $\mathbb{C}_{+}$. In particular, for every $z=x+iy \in \mathbb{C}_{+}$, we have
\begin{equation}\label{eq:uasoqaqan}
|(P_y*f)(x)| \lesssim \|(P_y*f)\|_{H^{\Phi}(\mathbb{R})}^{\mathrm{lux}}
\lesssim \|f\|_{H^{\Phi}(\mathbb{R})}^{\mathrm{lux}},
\end{equation}
by Proposition~\ref{pro:mainqaq2aop}.

Moreover, by Theorem~\ref{pro:main5aq5}, the Szeg\"o projection $\mathcal{P}_S(f)$ belongs to $H^{\Phi}(\mathbb{C}_{+})$ and satisfies
\[
\|\mathcal{P}_S(f)\|_{H^{\Phi}(\mathbb{C}_{+})}^{\mathrm{lux}}
\lesssim \|f\|_{H^{\Phi}(\mathbb{R})}^{\mathrm{lux}}.
\]
Consequently, using Proposition~\ref{pro:main6mq4}, we obtain
\begin{equation}\label{eq:uasoqaqn}
|\mathcal{P}_S(f)(z)| \leq \Phi^{-1}\left(\frac{2}{\pi y}\right) \|\mathcal{P}_S(f)\|_{H^{\Phi}(\mathbb{C}_{+})}^{\mathrm{lux}}
\lesssim \Phi^{-1}\left(\frac{2}{\pi y}\right) \|f\|_{H^{\Phi}(\mathbb{R})}^{\mathrm{lux}},
\end{equation}
for all $z=x+iy \in \mathbb{C}_{+}$.

Let $z=x+iy \in \mathbb{C}_{+}$. For every $t \in \mathbb{R}$, we have the identity
\[
\frac{y}{(t-x)^2+y^2}
=
\frac{i}{2}\left(\frac{1}{t-\overline{z}}-\frac{1}{t-z}\right).
\]
Multiplying by $f(t)$ and integrating over $\mathbb{R}$, we obtain
\[
\int_{\mathbb{R}}\frac{f(t)}{t-\overline{z}}\,dt
=
-2i\pi (P_y*f)(x)
+
2i\pi \mathcal{P}_S(f)(z).
\]
Combining this identity with estimates \eqref{eq:uasoqaqan} and \eqref{eq:uasoqaqn}, we conclude that
\[
\left|\int_{\mathbb{R}}\frac{f(t)}{t-\overline{z}}\,dt\right|
\leq C_{\Phi, z} \|f\|_{H^{\Phi}(\mathbb{R})}^{\mathrm{lux}},
\]
which completes the proof.
\end{proof}

\medskip

\noindent
We now turn to a fundamental representation result, showing that elements of $H^{\Phi}(\mathbb{C}_{+})$ can be described in terms of boundary distributions via the Poisson integral.

\begin{theorem}\label{pro:main6mmq4}
Let $\Phi \in \mathscr{L}$ and let $F$ be an analytic function on $\mathbb{C}_{+}$. 
The following assertions are equivalent:
\begin{itemize}
\item[(i)] $F \in H^{\Phi}(\mathbb{C}_{+})$;
\item[(ii)] There exists a bounded tempered distribution $f$ on $\mathbb{R}$ such that 
$f \in H^{\Phi}(\mathbb{R})$ and
\[
F(x+iy)=(P_{y}\ast f)(x), \qquad \forall\, x+iy \in \mathbb{C}_{+}.
\]
\end{itemize}
Moreover,
\[
\|F\|_{H^{\Phi}(\mathbb{C}_{+})}^{lux} \approx 
\|f\|_{H^{\Phi}(\mathbb{R})}^{lux}
\]
and
\[
\lim_{y\to 0}\|F(\cdot+iy)-f\|_{H^{\Phi}(\mathbb{R})}^{lux}=0.
\]
\end{theorem}

\begin{proof}
\noindent
$(ii)\Rightarrow(i)$.
This implication follows directly from Proposition~\ref{pro:mainqaq2aop}.

\medskip

\noindent
$(i)\Rightarrow(ii)$.
Assume that $F\not\equiv 0$, since the result is trivial when $F\equiv 0$.
Let $z_{0}=x_{0}+iy_{0}\in\mathbb{C}_{+}$ and consider a sequence 
$(y_{n})_{n}$ such that $0<y_{n}<y_{0}$ and $y_{n}\to 0$. 
For $z\in\mathbb{C}_{+}$, set
\[
F_{n}(z)=F(z+i y_{n}).
\]
Let $z=x+iy\in\mathbb{C}_{+}$. By Proposition~\ref{pro:main6mq4}, we have
\[
|F_{n}(z)|
\leq 
\Phi^{-1}\!\left(\frac{2}{\pi (y+y_{n})}\right)
\|F\|_{H^{\Phi}(\mathbb{C}_{+})}^{lux}
\leq 
\Phi^{-1}\!\left(\frac{2}{\pi y_{n}}\right)
\|F\|_{H^{\Phi}(\mathbb{C}_{+})}^{lux}.
\]
Hence $F_{n}\in H^{\infty}(\mathbb{C}_{+})$. 
Therefore there exists a unique function $f_{n}\in L^{\infty}(\mathbb{R})$ such that
\begin{equation}\label{eq:ualp5son}
F_{n}(x+iy)=(P_{y}\ast f_{n})(x),
\qquad \forall\, x+iy\in\mathbb{C}_{+},
\end{equation}
and $\|F_{n}\|_{H^{\infty}}=\|f_{n}\|_{L^{\infty}}$, (see \cite[Theorem 11.6]{javadmas}).  Moreover,
\[
f_{n}(x)
=\lim_{y\to 0}(P_{y}\ast f_{n})(x)
=\lim_{y\to 0}F(x+i(y+y_{n}))
=F(x+i y_{n})
\]
for almost every $x\in\mathbb{R}$.

Let $\varphi\in\mathscr{S}(\mathbb{R})$. 
Since $\mathscr{S}(\mathbb{R})\subset L^{1}(\mathbb{R})$, we obtain
\[
|(\varphi\ast f_{n})(x)|
\le 
\int_{\mathbb{R}}|\varphi(x-y)f_{n}(y)|\,dy
\le 
\|f_{n}\|_{L^{\infty}}\|\varphi\|_{L^{1}}.
\]
Thus $(\varphi\ast f_{n})\in L^{\infty}(\mathbb{R})$, which shows that $f_{n}$ defines a bounded tempered distribution on $\mathbb{R}$.

Let $t\in\mathbb{R}$ and $\alpha>0$. 
For $z=x+iy\in\Gamma_{\alpha}(t)$, we have $z+i y_{n}\in\Gamma_{\alpha}(t)$ since
\[
|x-t|<\alpha y
\quad\Rightarrow\quad
|x-t|<\alpha(y+y_{n}).
\]
Consequently,
\[
(F_{n})_{\alpha}^{*}(t)
=\sup_{z\in\Gamma_{\alpha}(t)}|F_{n}(z)|
=\sup_{z\in\Gamma_{\alpha}(t)}|F(z+i y_{n})|
\le 
\sup_{\zeta\in\Gamma_{\alpha}(t)}|F(\zeta)|
=F_{\alpha}^{*}(t).
\]
It follows that
\begin{equation}\label{eq:ualp5soaqn}
(f_{n})_{P}^{*}(t)
=\sup_{y>0}|(P_{y}\ast f_{n})(t)|
=\sup_{y>0}|F_{n}(t+iy)|
\le 
(F_{n})_{\alpha}^{*}(t)
\le 
F_{\alpha}^{*}(t).
\end{equation}
Since $F\in H^{\Phi}(\mathbb{C}_{+})$, we have $F_{\alpha}^{*}\in L^{\Phi}(\mathbb{R})$ by Lemma~\ref{pro:main6aapaqmaqq4k}. 
Combining this with \eqref{eq:ualp5soaqn}, we deduce that $(f_{n})_{P}^{*}\in L^{\Phi}(\mathbb{R})$. 
Hence $f_{n}\in H^{\Phi}(\mathbb{R})$ and
\begin{equation}\label{eq:ualp5soaqn1}
\|f_{n}\|_{H^{\Phi}(\mathbb{R})}^{lux}
\lesssim
\|F\|_{H^{\Phi}(\mathbb{C}_{+})}^{lux}.
\end{equation}
Thus $(f_{n})_{n}$ is bounded in $H^{\Phi}(\mathbb{R})$. 
Consequently, there exist $f\in\mathscr{S}'(\mathbb{R})$ and a subsequence $(f_{n_{k}})_{k}$ such that $f_{n_{k}}\to f$ in $\mathscr{S}'(\mathbb{R})$, (see  \cite{yangLiangKy}).  In particular, for all $\varphi\in\mathscr{S}(\mathbb{R})$,
\[
\lim_{k\to\infty}\langle f_{n_{k}},\varphi\rangle
=
\langle f,\varphi\rangle.
\]
Let $x+iy\in\mathbb{C}_{+}$. 
Since $P_{y}(x-\cdot)\in L^{1}(\mathbb{R})$ and the Schwartz space $\mathscr{S}(\mathbb{R})$ is dense in $L^{1}(\mathbb{R})$, we deduce that
\begin{equation}\label{eq:ualpn1}
F(x+iy)
=\lim_{k\to\infty}(P_{y}\ast f_{n_{k}})(x)
=\lim_{k\to\infty}\langle P_{y}(x-\cdot),f_{n_{k}}\rangle
=\langle P_{y}(x-\cdot),f\rangle
=(P_{y}\ast f)(x).
\end{equation}
Moreover, by Fatou's lemma and \eqref{eq:ualp5soaqn1},
\begin{equation}\label{eq:ualp5sn1}
\|f\|_{H^{\Phi}(\mathbb{R})}^{lux}
\lesssim
\liminf_{k\to\infty}\|f_{n_{k}}\|_{H^{\Phi}(\mathbb{R})}^{lux}
\lesssim
\|F\|_{H^{\Phi}(\mathbb{C}_{+})}^{lux}.
\end{equation}
Since $f$ is a bounded distribution belonging to $H^{\Phi}(\mathbb{R})$, Proposition~\ref{pro:mainqaq2aop} yields
\begin{equation}\label{eq:ualp5snaq1}
\|F\|_{H^{\Phi}(\mathbb{C}_{+})}^{lux}
\lesssim
\sup_{y>0}\|(P_{y}\ast f)\|_{H^{\Phi}(\mathbb{R})}^{lux}
\lesssim
\|f\|_{H^{\Phi}(\mathbb{R})}^{lux},
\end{equation}
and
\[
\lim_{y\to 0}
\|F(\cdot+iy)-f\|_{H^{\Phi}(\mathbb{R})}^{lux}
=
\lim_{y\to0}
\|(P_{y}\ast f)-f\|_{H^{\Phi}(\mathbb{R})}^{lux}
=0.
\]

The proof is complete.
\end{proof}

\medskip

%\noindent
%This representation theorem establishes a precise correspondence between boundary Hardy--Orlicz spaces and analytic functions in the upper half-plane.

\begin{corollaire}\label{pro:main5aqaqplq5}
Let $\Phi \in \mathscr{L}$ and let $F \in H^{\Phi}(\mathbb{C}_{+})$. 
Then, for all $y>0$, the function $F(\cdot+iy)$ defines a bounded tempered distribution on $\mathbb{R}$ which belongs to $H^{\Phi}(\mathbb{R})$. Moreover, there exists a constant $C>0$, independent of $F$ and $y$, such that
\begin{equation}\label{eq:deaqa2}
\|F(\cdot+iy)\|_{H^{\Phi}(\mathbb{R})}^{lux}
\leq 
C\,\|F\|_{H^{\Phi}(\mathbb{C}_{+})}^{lux}.
\end{equation}
\end{corollaire}

\begin{proof}
Let $F\in H^{\Phi}(\mathbb{C}_{+})$. By Theorem \ref{pro:main6mmq4}, there exists a bounded tempered distribution $f$ on $\mathbb{R}$ such that $f\in H^{\Phi}(\mathbb{R})$ and
\[
F(x+iy)=(P_y*f)(x), \qquad \text{for all } x+iy\in\mathbb{C}_{+}.
\]
Moreover, $\|F\|_{H^{\Phi}(\mathbb{C}_{+})}^{lux}\approx 
\|f\|_{H^{\Phi}(\mathbb{R})}^{lux}$. Consequently, for every $y>0$,
\[
\|F(\cdot+iy)\|_{H^{\Phi}(\mathbb{R})}^{lux}
=
\|P_y*f\|_{H^{\Phi}(\mathbb{R})}^{lux}
\lesssim 
\|f\|_{H^{\Phi}(\mathbb{R})}^{lux}
\approx 
\|F\|_{H^{\Phi}(\mathbb{C}_{+})}^{lux},
\]
where the first inequality follows from Proposition \ref{pro:mainqaq2aop}.

Let $\varphi\in\mathscr{S}(\mathbb{R})$. Since $\mathscr{S}(\mathbb{R})\subset L^{1}(\mathbb{R})$, we have for every $x\in\mathbb{R}$,
\[
|(\varphi*F(\cdot+iy))(x)|
=
\left|
\int_{\mathbb{R}}\varphi(x-u)F(u+iy)\,du
\right|
\le 
\int_{\mathbb{R}}|\varphi(x-u)||F(u+iy)|\,du.
\]

By Proposition \ref{pro:main6mq4}, we obtain
\[
|(\varphi*F(\cdot+iy))(x)|
\le 
\Phi^{-1}\!\left(\frac{2}{\pi y}\right)
\|F\|_{H^{\Phi}(\mathbb{C}_{+})}^{lux}
\|\varphi\|_{L^{1}(\mathbb{R})}.
\]
Hence $(\varphi*F(\cdot+iy))\in L^{\infty}(\mathbb{R})$. This shows that $F(\cdot+iy)$ defines a bounded tempered distribution on $\mathbb{R}$.
\end{proof}

Let us prove the following lemma that will needed below.
\begin{lemme}\label{pro:main5aqq5}
Let  $\beta >0$ and  $\Phi \in \mathscr{L}$. For  $F\in H^{\Phi}(\mathbb{C_{+}})$, we have 
$$  \int_{\mathbb{R}}\frac{F(t+i\beta)}{t-\overline{z}}dt=0, ~~\forall~z \in \mathbb{C}_{+}.
  $$
\end{lemme}

\begin{proof}
Since $F\in H^{\Phi}(\mathbb{C}_{+})$, Corollary~\ref{pro:main5aqaqplq5} ensures that 
$F(\cdot+i\beta)$ defines a bounded tempered distribution on $\mathbb{R}$ and belongs to $H^{\Phi}(\mathbb{R})$.
Consequently,
\begin{equation}\label{eq:uason}
\left|\int_{\mathbb{R}}\frac{F(t+i\beta)}{t-\overline{z}}\,dt\right|<\infty,
\end{equation}
thanks to Corollary \ref{pro:main5aqpq5}. 

\medskip

We now prove that the quantity in \eqref{eq:uason} is equal to $0$.

Let $(O,\overrightarrow{u},\overrightarrow{v})$ be the canonical orthonormal frame of the complex plane. 
For $R > \max\{1, \beta+|z|\}$, consider $\Gamma_{R,\beta}$ the curve defined by 
 $$ \Gamma_{R,\beta}:=\{ \omega \in \mathbb{C}_{+}:\hspace*{0.25cm} |\omega-i\beta| \leq R \hspace*{0.25cm}\text{and} \hspace*{0.25cm} \mathrm{Im}(\omega) \geq \beta       \}.       $$
Consider the  triangle $AOB$ rectangular at $O$ and contained in the frame  $(O,\overrightarrow{u},\overrightarrow{v})$ such that the distances $AO=\beta$ and $AB=R$, and put $\theta_{R}=mes(\widehat{OBA})$, the measure of the angle $OBA$. 

\medskip

By applying reasoning similar to that used in the proof of  \cite[Lemma 2.7]{djefeuto},  we obtain
$$  \left|\int_{-R\cos \theta_{R}}^{R\cos \theta_{R}}\frac{F(t+i\beta)}{t-\overline{z}}dt \right|
\lesssim \frac{R}{R- |i\beta+\overline{z}|}\|F\|_{H^{\Phi}(\mathbb{C}_{+})}^{lux} \int_{\theta_{R}}^{\pi/2} \Phi^{-1}\left(\frac{1}{ R\theta}\right)d\theta.  $$

\medskip

Suppose that $\Phi$ is of both lower type $p$ and upper type $q$ such that $0<p\leq q \leq 1$. 

\medskip

Since the function $t\mapsto \frac{\Phi^{-1}(t)}{t^{1/q}}$ is non-decreasing on $(0, \infty)$, we deduce that
$$  \int_{\theta_{R}}^{\pi/2} \Phi^{-1}\left(\frac{1}{ R\theta}\right)d\theta \lesssim  \beta\Phi^{-1}\left(\frac{1}{ \beta}\right)\int_{\beta/R}^{\pi/2} \left(\frac{1}{ R\theta}\right)^{1/q} d\theta.  $$
We have
\[ \int_{\beta/R}^{\pi/2} \theta^{-\frac{1}{q}} d\theta=
\left\{
\begin{array}{ll} \ln\left( \frac{\pi R}{2 \beta} \right) & \mbox{ if $q=1$ }\\\\ \frac{q}{1-q}[ (2/\pi)^{\frac{1}{q}-1}  - (R/\beta)^{\frac{1}{q}-1}] & \mbox{ if $q<1$}
\end{array}
\right. \]
We deduce that
\[ \left|\int_{-R\cos \theta_{R}}^{R\cos \theta_{R}}\frac{F(t+i\beta)}{t-\overline{z}}dt \right|\lesssim
\left\{
\begin{array}{ll}\beta\Phi^{-1}\left(\frac{1}{ \beta}\right) \|F\|_{H^{\Phi}}^{lux} \times\frac{R}{R- |i\beta+\overline{z}|} \times\frac{\ln\left( \frac{\pi R}{2 \beta}\right)}{R} & \mbox{ if $q=1$ }\\\\ \beta\Phi^{-1}\left(\frac{1}{ \beta}\right) \|F\|_{H^{\Phi}}^{lux} \times\frac{R}{R- |i\beta+\overline{z}|} \times \left[ \frac{(2/\pi)^{\frac{1}{q}-1}}{R^{1/q}}  - \frac{(1/\beta)^{\frac{1}{q}-1}}{R}\right]& \mbox{ if $q<1$}
\end{array}
\right. \]
It follows that
$$ \left|\int_{-R\cos \theta_{R}}^{R\cos \theta_{R}}\frac{F(t+i\beta)}{t-\overline{z}}dt \right|  \longrightarrow 0, 
        $$
when $R\longrightarrow \infty$.    The proof is complete.     
\end{proof}

The following result gives each element in the holomorphic Hardy-Orlicz space as the Szeg\"o projection of a function in the Hardy-Orlicz space of the real line.
\begin{theorem}\label{pro:main6mmq4p}
Let $\Phi \in \mathscr{L}$ and let $F$ be an analytic function on $\mathbb{C}_{+}$. The following assertions are equivalent:
\begin{itemize}
\item[(i)] $F \in H^{\Phi}(\mathbb{C}_{+})$;
\item[(ii)] There exists a bounded tempered distribution $f$ on $\mathbb{R}$ such that $f \in H^{\Phi}(\mathbb{R})$ and
\[
F(z)= (P_{y} \ast f)(x)= \mathcal{P}_{S}(f)(z), \qquad \forall\, z=x+iy \in \mathbb{C}_{+}.
\]
\end{itemize}
\end{theorem}

\begin{proof}
\noindent
$(ii)\Rightarrow(i)$. 
This implication follows directly from the boundedness properties of the Poisson operator and Szeg\"o projection  on $H^{\Phi}(\mathbb{R})$ (see Proposition \ref{pro:mainqaq2aop} and  Theorem~\ref{pro:main5aq5}).

\medskip

\noindent
$(i)\Rightarrow(ii)$. 
Assume that $F \in H^{\Phi}(\mathbb{C}_{+})$. By Theorem~\ref{pro:main6mmq4}, there exists a bounded tempered distribution $f$ on $\mathbb{R}$ such that $f \in H^{\Phi}(\mathbb{R})$ and
\[
F(x+iy)=(P_{y}\ast f)(x), \qquad \forall\, x+iy \in \mathbb{C}_{+},
\]
with
\[
\lim_{y\to 0}\|F(\cdot+iy)-f\|_{H^{\Phi}(\mathbb{R})}^{\mathrm{lux}}=0.
\]
Let $z \in \mathbb{C}_{+}$ and $\beta>0$. By Corollary~\ref{pro:main5aqpq5}, we have
\[
\left|\int_{\mathbb{R}}\frac{f(t)}{t-\overline{z}}\,dt
-
\int_{\mathbb{R}}\frac{F(t+i\beta)}{t-\overline{z}}\,dt\right|
\lesssim \|F(\cdot+i\beta)-f\|_{H^{\Phi}(\mathbb{R})}^{\mathrm{lux}}.
\]
Letting $\beta \to 0$ and using the above convergence, we obtain
\[
\int_{\mathbb{R}}\frac{f(t)}{t-\overline{z}}\,dt
=
\lim_{\beta \to 0}\int_{\mathbb{R}}\frac{F(t+i\beta)}{t-\overline{z}}\,dt.
\]
By Lemma~\ref{pro:main5aqq5}, the right-hand side is equal to $0$, hence
\[
\int_{\mathbb{R}}\frac{f(t)}{t-\overline{z}}\,dt = 0.
\]
Therefore, 
\[
F(z)= (P_{y} \ast f)(x)= \mathcal{P}_{S}(f)(z), \qquad \forall\, z \in \mathbb{C}_{+}.
\]
This completes the proof.
\end{proof}

We then obtain the following properties of the Szeg\"o projection.
\begin{theorem}\label{pro:main2aaqaqop}
Let $\Phi \in \mathscr{L}$. The Szeg\"o projection 
$\mathcal{P}_{S}: H^{\Phi}(\mathbb{R}) \longrightarrow H^{\Phi}(\mathbb{C}_{+})$ is a linear, continuous, and surjective operator.
\end{theorem}

\begin{proof}
The result follows directly from Theorems~\ref{pro:main5aq5} and~\ref{pro:main6mmq4p}.
\end{proof}

Let $(z_{n})_{n\geq 1}$ be a sequence of elements of 
 $\mathbb{C}_{+}$ such that  $ \sum_{n=1}^{\infty}\frac{\mathrm{Im}( z_{n})}{|i+z_{n}|^{2}}< \infty$. The Blaschke product on $\mathbb{C}_{+}$ associated to $(z_{n})_{n\geq 1}$ is the function $B$ defined by
\begin{equation}\label{eq:phipaqm6e}
B(z)=\prod_{n=1}^{\infty}\dfrac{|\frac{i-z_{n}}{i-\overline{z_{n}}}|}{\frac{i -z_{n}}{i-\overline{z_{n}}}}\frac{z-z_{n}}{z-\overline{z_{n}}}, ~~\forall~
z \in \mathbb{C}_{+}.
\end{equation}

\medskip

Let $\sigma$ be a positive and singular Borel measure on $\mathbb{R}$ such that
$  \int_{\mathbb{R}}\frac{d\sigma(t)}{1+t^{2}}< \infty$
 and consider $S_{\sigma}$ the function defined by
$$ S_{\sigma}(z)=\exp\left\{ \frac{i}{\pi}\int_{\mathbb{R}}\left(\frac{1}{z-t}+\frac{t}{1+t^{2}}\right)d\sigma(t)  \right\}, ~~\forall~z \in \mathbb{C}_{+}.
   $$

Let $\alpha \geq 0$ be a real and put
$$  E_{\alpha}(z)=\exp\left\{i\alpha z  \right\}, ~~\forall~z \in \mathbb{C}_{+}.
  $$

Let $F$ be an analytic function on $\mathbb{C_{+}}$. We say that $F$ is an inner function on $\mathbb{C_{+}}$ if $F\in H^{\infty}(\mathbb{C_{+}})$
and  $ \left|\lim_{y\to 0}F(t+iy)\right|=1,$
for almost all $t\in \mathbb{R}$.

\begin{remark}\label{pro:maiaqq5}
The Blaschke product $B$ and the functions  $S_{\sigma}$ and $E_{\alpha}$ are inner functions on $\mathbb{C_{+}}$.
\end{remark}

Let $f$ be a positive measurable function on   $\mathbb{R}$ such that  $\log f \in L^{1}\left(\mathbb{R}, \frac{dt}{1+t^{2}}\right)$. The outer function associated with $f$ is  defined by
$$  O_{f}(z)= \exp\left\{ \frac{i}{\pi}\int_{\mathbb{R}}\left(\frac{1}{z-t}+\frac{t}{1+t^{2}}\right)\log|f(t)|dt  \right\}  , ~~\forall~z \in \mathbb{C}_{+}. 
  $$
Let us recall the following result.
\begin{theorem}[Theorem 3.24, \cite{djesehb1}]\label{pro:mainfaaqaqaqq5}
Let $\Phi \in \mathscr{L} \cup \mathscr{U}$ and let $0 \not\equiv F \in H^{\Phi}(\mathbb{C}_{+})$. Then there exists a unique function 
$f \in L^{\Phi}(\mathbb{R})$ such that $\log |f| \in L^{1}\!\left(\mathbb{R}, \frac{dt}{1+t^{2}}\right)$,
\[
f(x)=\lim_{y\to 0}F(x+iy), \quad \text{for a.e. } x \in \mathbb{R},
\]
and $f(t)\neq 0$ for almost every $t \in \mathbb{R}$. Moreover, for every $z=x+iy \in \mathbb{C}_{+}$, we have
\[
\log|F(x+iy)| \leq \frac{1}{\pi}\int_{\mathbb{R}}\frac{y}{(x-t)^{2}+y^{2}}\log|f(t)|\,dt.
\]
Furthermore,
\[
\|F\|_{H^{\Phi}(\mathbb{C}_{+})}^{\mathrm{lux}}
= \lim_{y\to 0}\|F(\cdot+iy)\|_{L^{\Phi}(\mathbb{R})}^{\mathrm{lux}}
= \|f\|_{L^{\Phi}(\mathbb{R})}^{\mathrm{lux}}.
\]
\end{theorem}

\medskip

\begin{theorem}\label{pro:main0apaaql0}
Let $\Phi \in \mathscr{L} \cup \mathscr{U}$ and let $0 \not\equiv F \in H^{\Phi}(\mathbb{C}_{+})$. Then there exists a unique function
$f \in L^{\Phi}(\mathbb{R})$ such that $\log |f| \in L^{1}\!\left(\mathbb{R}, \frac{dt}{1+t^{2}}\right)$
and
$$f(x)=\lim_{y\to 0}F(x+iy), $$
for almost every $x \in \mathbb{R}$. Moreover, there exists a unique positive singular Borel measure $\sigma$ on $\mathbb{R}$ such that
\[
\int_{\mathbb{R}}\frac{d\sigma(t)}{1+t^{2}}<\infty.
\]
Furthermore, $F$ admits a unique factorization of the form
\begin{equation}\label{eq:phiaq6ql3de}
F(z) = E_{\alpha}(z)\, B(z)\, S_{\sigma}(z)\, O_{|f|}(z), 
\qquad \forall\, z \in \mathbb{C}_{+},
\end{equation}
and
\[
\|F\|_{H^{\Phi}(\mathbb{C}_{+})}^{\mathrm{lux}}
= \|O_{|f|}\|_{H^{\Phi}(\mathbb{C}_{+})}^{\mathrm{lux}}
= \|f\|_{L^{\Phi}(\mathbb{R})}^{\mathrm{lux}}.
\]
\end{theorem}

\medskip

For a detailed proof of Theorem~\ref{pro:main0apaaql0}, we refer the reader to \cite[Theorem 13.15]{javadmas}, as the argument follows the same lines.

\subsection{Campanato--Orlicz spaces}

\begin{theorem}[Theorem 5.2.1, \cite{yangLiangKy}]\label{pro:main2aaoamdp}
Let $\Phi \in \mathscr{L}$ and let $s \in \mathbb{N}$ satisfy $s \geq m_{\Phi}$, where $m_{\Phi}$ is defined by \eqref{eq:dqtap2}. Then the dual space of $H^{\Phi}(\mathbb{R})$, denoted by $(H^{\Phi}(\mathbb{R}))^*$, is the Campanato--Orlicz space $\mathcal{L}^{\Phi,s}(\mathbb{R})$ in the following sense: for every continuous linear functional $L \in (H^{\Phi}(\mathbb{R}))^*$, there exists a unique  $b \in \mathcal{L}^{\Phi,s}(\mathbb{R})$ such that
\[
L(f)=\int_{\mathbb{R}} f(x)\,\overline{b(x)}\,dx,
\]
for every $f \in H^{\Phi}(\mathbb{R})$. Moreover,
\[
\|b\|_{\mathcal{L}^{\Phi,s}(\mathbb{R})}
\approx
\|L\|_{(H^{\Phi}(\mathbb{R}))^*},
\]
where the implicit positive constants are independent of $b$.
\end{theorem}

\medskip
\noindent

The following proposition shows that $\mathcal{L}^{\Phi,s}(\mathbb{R})$ is a natural intermediate space between Schwartz test functions space and its dual.

\begin{proposition}\label{pro:main2aaoammqdp1}
Let $\Phi \in \mathscr{L}$ and let $s \in \mathbb{N}$ with $s \geq m_{\Phi}$, where $m_{\Phi}$ is defined in \eqref{eq:dqtap2}. Then
\[
\mathscr{S}(\mathbb{R}) \subset \mathcal{L}^{\Phi,s}(\mathbb{R}) \subset \mathscr{S}'(\mathbb{R}).
\]
\end{proposition}

\begin{proof}
Since $(H^{\Phi}(\mathbb{R}))^* = \mathcal{L}^{\Phi,s}(\mathbb{R})$ and $\mathscr{S}(\mathbb{R}) \subset H^{\Phi}(\mathbb{R})$, any $f \in \mathcal{L}^{\Phi,s}(\mathbb{R})$ defines a continuous linear functional on $\mathscr{S}(\mathbb{R})$. Thus,
\[
\mathcal{L}^{\Phi,s}(\mathbb{R}) \subset \mathscr{S}'(\mathbb{R}).
\]

Conversely, let $\varphi \in \mathscr{S}(\mathbb{R})$ and define
\[
T_{\varphi}(f) := \int_{\mathbb{R}} f(x)\varphi(x)\,dx, \quad f \in H^{\Phi}(\mathbb{R}).
\]
By Proposition~\ref{pro:maaqqaq5}, there exists a constant $C>0$, independent of $\varphi$, such that
\[
|T_{\varphi}(f)| \leq C\,\mathcal{N}_{m_{\Phi}}(\varphi)\,\|f\|_{H^{\Phi}(\mathbb{R})}^{lux},
\quad \forall f \in H^{\Phi}(\mathbb{R}).
\]
Then $T_{\varphi}$ is a continuous linear functional on $H^{\Phi}(\mathbb{R})$, and hence $T_{\varphi} \in (H^{\Phi}(\mathbb{R}))^*$. By duality, there exists $g \in \mathcal{L}^{\Phi,s}(\mathbb{R})$ such that
\[
T_{\varphi}(f) = \int_{\mathbb{R}} f(x)g(x)\,dx, \quad \forall f \in H^{\Phi}(\mathbb{R}).
\]
This implies $\varphi = g$ almost everywhere, and therefore $\varphi \in \mathcal{L}^{\Phi,s}(\mathbb{R})$.

\medskip

The proof is complete.
\end{proof}

  \begin{proposition}\label{pro:main2aaoammqdp}
  Let $\Phi \in \mathscr{L}$ and let $s \in \mathbb{N}$ be such that $s \geq m_{\Phi}$, where $m_{\Phi}$ is defined as in \eqref{eq:dqtap2}. For all $f \in \mathcal{L}^{\Phi,s}(\mathbb{R})$ and $y>0$, we have $(P_{y}\ast f) \in \mathcal{L}^{\Phi,s}(\mathbb{R})$ and 
  \begin{equation}\label{eq:dqtaqmhqmp2}
   \|(P_{y}\ast f)\|_{\mathcal{L}^{\Phi,s}(\mathbb{R})} \leq C\, \| f\|_{\mathcal{L}^{\Phi,s}(\mathbb{R})},
  \end{equation}
  where $C$ is a positive constant independent of $f$ and $y$.
  \end{proposition}
  
  \begin{proof}
  Let $g \in H^{\Phi}(\mathbb{R})$. By the atomic decomposition of $H^{\Phi}(\mathbb{R})$, there exists a sequence of $(\Phi,s)$-atoms $(a_j)_j$ such that
  \[
  f=\sum_{j} a_j \quad \text{in } \mathscr{S}'(\mathbb{R}),
  \quad\text{and}\quad
  \|f\|_{H^{\Phi}(\mathbb{R})}^{\mathrm{lux}}
  \approx 
  \|f\|_{H^{\Phi,s}_{\mathrm{at}}(\mathbb{R})}^{\mathrm{lux}}.
  \]
 By definition, each atom $a_j$ belongs to $L^{2}(\mathbb{R})$ and satisfies $\supp(a_j)\subset I_j$, where $I_j$ is a bounded interval. In particular, $a_j \in L^1(\mathbb{R})$. Since $P_y \in L^1(\mathbb{R})$, $a_j \in L^1(\mathbb{R})$, and $f \in L^1_{\mathrm{loc}}(\mathbb{R})$, the function $(x,t) \mapsto P_y(x-t)f(t)a_j(x)$ is integrable on $\mathbb{R}^2$. Hence, by Fubini's theorem, we can interchange the order of integration and obtain
  \[
  \langle P_y * f, a_j \rangle
  = \int_{\mathbb{R}} \int_{\mathbb{R}} P_y(x-t) f(t) a_j(x)\,dt\,dx
  = \int_{\mathbb{R}} f(t) \left( \int_{\mathbb{R}} P_y(x-t) a_j(x)\,dx \right) dt.
  \]
  Recognizing the convolution, we get
  \[
  \langle P_y * f, a_j \rangle = \int_{\mathbb{R}} f(t) (P_y * a_j)(t)\,dt = \langle f, P_y * a_j \rangle.
  \]
 Since the series $\sum_j a_j$ converges in $H^{\Phi}(\mathbb{R})$, and convolution with $P_y \in L^1(\mathbb{R})$ is continuous on $H^{\Phi}(\mathbb{R})$ (see Proposition \ref{pro:mainqaq2aop}), we have
  \[
  P_y * g = \sum_j P_y * a_j \quad \text{in } H^{\Phi}(\mathbb{R}).
  \]
  Therefore,
  \[
  \langle P_y * f, g \rangle
  = \sum_j \langle P_y * f, a_j \rangle
  = \sum_j \langle f, P_y * a_j \rangle
  = \left\langle f, \sum_j P_y * a_j \right\rangle
  = \langle f, P_y * g \rangle.
  \]
Now, since $f \in \mathcal{L}^{\Phi,s}(\mathbb{R}) = (H^{\Phi}(\mathbb{R}))^*$, we obtain by duality
  \[
  |\langle P_y * f, g \rangle|
  = |\langle f, P_y * g \rangle|
  \le C\, \|f\|_{\mathcal{L}^{\Phi,s}(\mathbb{R})} \|P_y * g\|_{H^{\Phi}(\mathbb{R})}^{lux}.
  \]
  Using again Proposition \ref{pro:mainqaq2aop}, we have
  \[
  \|P_y * g\|_{H^{\Phi}(\mathbb{R})}^{lux} \le C \|g\|_{H^{\Phi}(\mathbb{R})}^{lux}.
  \]
  Hence,
  \[
  |\langle P_y * f, g \rangle|
  \le C\, \|f\|_{\mathcal{L}^{\Phi,s}(\mathbb{R})} \|g\|_{H^{\Phi}(\mathbb{R})}^{lux}.
  \]
Finally, taking the supremum over all $g \in H^{\Phi}(\mathbb{R})$ such that $\|g\|_{H^{\Phi}(\mathbb{R})}^{lux} \le 1$, we conclude that
  \[
  \|P_y * f\|_{\mathcal{L}^{\Phi,s}(\mathbb{R})}
  \approx \sup_{\|g\|_{H^{\Phi}} \le 1} |\langle P_y * f, g \rangle|
  \le C \|f\|_{\mathcal{L}^{\Phi,s}(\mathbb{R})}.
  \]
  
  This proves that $P_y * f \in \mathcal{L}^{\Phi,s}(\mathbb{R})$ and completes the proof.
  \end{proof}

\begin{lemme}\label{pro:mainq5amaq}
Let $\Phi \in \mathscr{L}$ and let $s\in\mathbb{N}$ with $s\ge m_{\Phi}$. 
If $f \in \mathcal{L}^{\Phi,s}(\mathbb{R})$, then $\mathcal{H}(f)\in \mathcal{L}^{\Phi,s}(\mathbb{R})$ and
\begin{equation}\label{eq:deaqaa21}
\|\mathcal{H}(f)\|_{\mathcal{L}^{\Phi,s}(\mathbb{R})}
\leq C\,\|f\|_{\mathcal{L}^{\Phi,s}(\mathbb{R})},
\end{equation}
where $C$ is independent of $f$.
\end{lemme}

\begin{proof}
Since $f \in \mathcal{L}^{\Phi,s}(\mathbb{R}) \subset \mathscr{S}'(\mathbb{R})$, the Hilbert transform $\mathcal{H}(f)$ is well-defined in $\mathscr{S}'(\mathbb{R})$ and satisfies
\[
\langle \mathcal{H}(f), \varphi \rangle = -\langle f, \mathcal{H}(\varphi)\rangle,
\quad \forall \varphi \in \mathscr{S}(\mathbb{R}).
\]
By density, this identity extends to all $g \in L^2(\mathbb{R})$.

Let $g \in H^{\Phi}(\mathbb{R})$. By the atomic decomposition of $H^{\Phi}(\mathbb{R})$, there exist $(\Phi,s)$-atoms $(a_j)_j$ such that
\[
g = \sum_j a_j \quad \text{in } H^{\Phi}(\mathbb{R}),
\quad\text{and}\quad
\|g\|_{H^{\Phi}}^{lux} \approx \|g\|_{H^{\Phi,s}_{at}}^{lux}.
\]
Since each $a_j \in L^2(\mathbb{R})$, we have
\[
\langle \mathcal{H}(f), a_j \rangle = -\langle f, \mathcal{H}(a_j)\rangle.
\]
By linearity and convergence of the atomic decomposition,
\[
\langle \mathcal{H}(f), g \rangle = -\langle f, \mathcal{H}(g)\rangle.
\]

Using the duality $(H^{\Phi}(\mathbb{R}))^{*}=\mathcal{L}^{\Phi,s}(\mathbb{R})$ and the boundedness of $\mathcal{H}$ on $H^{\Phi}(\mathbb{R})$ (see Theorem~\ref{pro:main5aq5}), we obtain
\[
|\langle \mathcal{H}(f), g \rangle|
= |\langle f, \mathcal{H}(g)\rangle|
\leq \|f\|_{\mathcal{L}^{\Phi,s}}\|\mathcal{H}(g)\|_{H^{\Phi}}^{lux}
\lesssim \|f\|_{\mathcal{L}^{\Phi,s}}\|g\|_{H^{\Phi}}^{lux}.
\]
Taking the supremum over $\|g\|_{H^{\Phi}}^{lux}\le 1$, we conclude that
\[
\|\mathcal{H}(f)\|_{\mathcal{L}^{\Phi,s}(\mathbb{R})}
\lesssim \|f\|_{\mathcal{L}^{\Phi,s}(\mathbb{R})}.
\]
\end{proof}

\medskip
\noindent
The following theorem gives properties Szeg\"o projection as an operator from $\mathcal{L}^{\Phi,s}(\mathbb{R})$ to its holomorphic counterpart..
%The surjectivity reflects the fact that every analytic function in $\mathcal{H}\mathcal{L}^{\Phi,s}(\mathbb{C}_{+})$ admits a boundary trace belonging to $\mathcal{L}^{\Phi,s}(\mathbb{R})$, thereby establishing a deep connection between real analysis and complex analysis.

\begin{theorem}\label{thm:szego_projection_surjective}
Let $\Phi \in \mathscr{L}$ and let $s \in \mathbb{N}$ be such that $s \ge m_{\Phi}$. 
Then the Szeg\"o projection
\[
\mathcal{P}_{S} : \mathcal{L}^{\Phi,s}(\mathbb{R}) \longrightarrow \mathcal{H}\mathcal{L}^{\Phi,s}(\mathbb{C}_{+})
\]
is a linear, continuous, and surjective operator.
\end{theorem}

\begin{proof}
The result follows directly from Proposition~\ref{pro:main2aaoammqdp} and Lemma~\ref{pro:mainq5amaq}.
\end{proof}

\section{Proof of the main results.} 
We prove in this section all our main results stated in the first section.

\subsection{Proof of Theorem \ref{pro:main2aqop}.}

\medskip

\noindent
Let us start by recalling the following result.

\begin{proposition}[Lemma 3, \cite{voltiko}]\label{pro:main 5aqaq2pl}
Let  $\Phi_{1}, \Phi_{2}$ and $\Phi_{3}$  be growth functions of the lower type. $L^{\Phi_{3}}=L^{\Phi_{1}}.L^{\Phi_{2}}$ if and only if  $\Phi_{3}^{-1} \sim \Phi_{1}^{-1}.\Phi_{2}^{-1}$,  where  $\Phi_{j}^{-1}$ is the inverse function of  $\Phi_{j}$, for  $j\in \{1,2,3\}$. 
\end{proposition}

The following result is an immediate consequence of Proposition \ref{pro:main 5aqaq2pl}. Therefore, the proof will be omitted.

\begin{lemme}\label{pro:main 5aqp1aqqq2pl}
Let  $\Phi_{1}, \Phi_{2}$ and $\Phi_{3}$  be growth functions of the lower type such that  $\Phi_{3}^{-1} \sim \Phi_{1}^{-1}.\Phi_{2}^{-1}$,  where  $\Phi_{j}^{-1}$ is the inverse function of  $\Phi_{j}$, for  $j\in \{1,2,3\}$.  For all  $F\in   H^{\Phi_{1}}(\mathbb{C}_{+})$ and  $G\in   H^{\Phi_{2}}(\mathbb{C}_{+})$, the product   $FG \in H^{\Phi_{3}}(\mathbb{C}_{+})$. Moreover,  
\begin{equation}\label{eq:eaamqq1}
\|FG\|_{H^{\Phi_{3}}}^{lux} \leq C \|F\|_{H^{\Phi_{1}}}^{lux}\|G\|_{H^{\Phi_{2}}}^{lux}, 
\end{equation}
where $C$ is constant independent of $F$ and $G$.  
\end{lemme}

We can now prove Theorem \ref{pro:main2aqop}.
\proof[Proof of Theorem \ref{pro:main2aqop}.]
Without loss of generality, we can assume that  $\Phi_{3}^{-1}= \Phi_{1}^{-1}.\Phi_{2}^{-1}$. 

\medskip

Let  $0\not\equiv F\in H^{\Phi_{3}}(\mathbb{C}_{+})$.  According to Theorem \ref{pro:main0apaaql0}, there exists a unique decomposition of the function $F$ of the form
$$  F(z) = E_{\alpha}(z)B(z)S_{\sigma}(z)O_{|f|}(z), ~~\forall~z\in \mathbb{C}_{+}.   $$
Moreover, 
  $ \|F\|_{H^{\Phi}}^{lux}= \left\|O_{|f|}\right\|_{H^{\Phi}}^{lux}=\|f\|_{L^{\Phi}}^{lux}.$ For $k \in \{1,2\}$, put   $$ f_{k}= \Phi_{k}^{-1}\circ\Phi_{3}\left(\|f\|_{L^{\Phi_{3}}}^{lux}\right)  \Phi_{k}^{-1}\circ\Phi_{3}\left(\frac{|f|}{\|f\|_{L^{\Phi_{3}}}^{lux}}\right).        $$
By reasoning in a similar way to that of the proof of \cite[Theorem 3]{djefeuto1}, we show that  $f_{k}\in  L^{\Phi_{k}}\left(\mathbb{R}\right)$ and $\log|f_{k}| \in  L^{1}\left(\mathbb{R},\frac{dt}{1+t^{2}}\right)$. Moreover,    
$$ \|f_{k}\|_{L^{\Phi_{k}}}^{lux} \leq \Phi_{k}^{-1}\circ\Phi_{3}\left(\|f\|_{L^{\Phi_{3}}}^{lux}\right).  $$
We deduce that, $O_{|f_{k}|} \in H^{\Phi_{k}}(\mathbb{C}_{+})$ and 
$\|O_{|f_{k}|}\|_{H^{\Phi_{k}}}^{lux}=\|f_{k}\|_{L^{\Phi_{k}}}^{lux}$. Since  $\Phi_{3}^{-1}= \Phi_{1}^{-1}.\Phi_{2}^{-1}$, it follows that $|f| = f_{1}f_{2}$ and $O_{|f_{1}|}.O_{|f_{2}|}=O_{|f|}$. Moreover, 
\begin{equation}\label{eq:inegaaqlitedehay}
\|O_{|f_{1}|}\|_{H^{\Phi_{1}}}^{lux}\|O_{|f_{2}|}\|_{H^{\Phi_{2}}}^{lux}  \leq \|O_{|f|}\|_{H^{\Phi_{3}}}^{lux}.
\end{equation}
For $z \in \mathbb{C}_{+}$, put  
$$  F_{1}(z) =O_{|f_{1}|}(z)   \hspace*{0.5cm}\textrm{and} \hspace*{0.5cm} F_{2}(z) =E_{\alpha}(z)B(z)S_{\sigma}(z)O_{|f_{2}|}(z).  $$
By construction, $F_{1}$ and $F_{2}$ are analytic functions on $\mathbb{C}_{+}$ such that  $F_{1} \in H^{\Phi_{1}}(\mathbb{C}_{+})$ and $F_{2} \in H^{\Phi_{2}}(\mathbb{C}_{+})$. Moreover, $\|F_{1}\|_{H^{\Phi_{1}}}^{lux}=\|f_{1}\|_{L^{\Phi_{1}}}^{lux}$ and $\|F_{2}\|_{H^{\Phi_{2}}}^{lux}=\|f_{2}\|_{L^{\Phi_{2}}}^{lux}$, since  $E_{\alpha}B S_{\sigma}$  is inner function on $\mathbb{C}_{+}$. It follows that, for all $z\in \mathbb{C}_{+}$, 
$$ F(z) =E_{\alpha}(z)B(z)S_{\sigma}(z)O_{|f_{2}|}(z)O_{|f_{1}|}(z)=   F_{1}(z) F_{2}(z)$$ 
and
$$  \|F_{1}\|_{H^{\Phi_{1}}}^{lux}\|F_{2}\|_{H^{\Phi_{2}}}^{lux}=\|f_{1}\|_{L^{\Phi_{1}}}^{lux}\|f_{2}\|_{L^{\Phi_{2}}}^{lux} \leq \|f\|_{L^{\Phi_{3}}}^{lux} = \|F\|_{H^{\Phi_{3}}}^{lux} \lesssim \|F_{1}\|_{H^{\Phi_{1}}}^{lux}\|F_{2}\|_{H^{\Phi_{2}}}^{lux},
  $$
thanks Lemma \ref{pro:main 5aqp1aqqq2pl}.
\epf

\subsection{Proof of Theorem \ref{pro:main6maqmq4}.} 

\medskip

\noindent
This result is about the atomic decomposition of holomorphic Hardy-Orlicz spaces. 
The key idea is to use the Poisson integral representation together with the atomic decomposition on the Hardy-Orlicz spaces of the real line.

\proof[Proof of Theorem \ref{pro:main6maqmq4}.]

\noindent $(ii) \Rightarrow (i)$ follows directly from Theorem~\ref{pro:main6mmq4p}.

\medskip

\noindent $(i) \Rightarrow (ii)$.
Let $F \in H^{\Phi}(\mathbb{C}_{+})$. Then there exists $f \in H^{\Phi}(\mathbb{R})$ such that
\[
F(z) = (P_y * f)(x) = \mathcal{P}_S(f)(z), \quad \forall z = x + iy \in \mathbb{C}_{+},
\]
and
\[
\|F\|_{H^{\Phi}(\mathbb{C}_{+})}^{\mathrm{lux}} \approx \|f\|_{H^{\Phi}(\mathbb{R})}^{\mathrm{lux}}, 
\quad 
\lim_{y \to 0} \|F(\cdot + iy) - f\|_{H^{\Phi}(\mathbb{R})}^{\mathrm{lux}} = 0,
\]
by Theorems~\ref{pro:main6mmq4} and~\ref{pro:main6mmq4p}.

Since $f \in H^{\Phi}(\mathbb{R})$, it admits an atomic decomposition $f = \sum_j a_j$ in $H^{\Phi,s}_{\mathrm{at}}(\mathbb{R})$ with
\[
\|f\|_{H^{\Phi}(\mathbb{R})}^{\mathrm{lux}} \approx \|f\|_{H^{\Phi,s}_{\mathrm{at}}}^{\mathrm{lux}}.
\]
Hence, for all $z \in \mathbb{C}_{+}$,
\[
F(z) = \mathcal{P}_S(f)(z) = \sum_{j} \mathcal{P}_S(a_j)(z).
\]
\epf

\subsection{Proof of Theorem \ref{pro:main2aaop}.} 

\medskip

\noindent
The following proposition identifies the boundary pairing associated with functions in Hardy--Orlicz spaces. 
It shows that the duality can be expressed either in the upper half-plane or on the boundary.

\begin{proposition}\label{pro:main2aaqaop}
Let $\Phi \in \mathscr{L}$ and let $s \in \mathbb{N}$ be such that $s \ge m_{\Phi}$. 
Let $F \in H^{\Phi}(\mathbb{C}_{+})$ and $G \in \mathcal{H}\mathcal{L}^{\Phi,s}(\mathbb{C}_{+})$. Define
\[
T_G(F) := \lim_{y \to 0} \int_{\mathbb{R}} F(x+iy)\,\overline{G(x+iy)}\,dx.
\]
Then the above limit exists and satisfies
\[
T_G(F) = \int_{\mathbb{R}} f(x)\,\overline{g(x)}\,dx.
\]
Moreover, there exists a positive constant $C$, independent of $f$ and $g$, such that
\begin{equation}\label{eq:ineitmlqpqqehay}
|T_G(F)| \le C\, \|G\|_{\mathcal{H}\mathcal{L}^{\Phi,s}(\mathbb{C}_{+})}\, \|F\|_{H^{\Phi}(\mathbb{C}_{+})}^{\mathrm{lux}}.
\end{equation}
\end{proposition}

\begin{proof}
Let $F \in H^{\Phi}(\mathbb{C}_{+})$ and $G \in \mathcal{H}\mathcal{L}^{\Phi,s}(\mathbb{C}_{+})$. 
Then there exist $f \in H^{\Phi}(\mathbb{R})$ and $g \in \mathcal{L}^{\Phi,s}(\mathbb{R})$ such that
\[
F(x+iy) = (P_y * f)(x), \qquad G(x+iy) = (P_y * g)(x) = \mathcal{P}_S(g)(x+iy),
\quad \forall x+iy \in \mathbb{C}_{+}.
\]
Moreover,
\[  \|F\|_{H^{\Phi}(\mathbb{C}_{+})}^{\mathrm{lux}} \approx \|f\|_{H^{\Phi}(\mathbb{R})}^{\mathrm{lux}},
 \qquad
\|G\|_{\mathcal{H}\mathcal{L}^{\Phi,s}(\mathbb{C}_{+})}
:= \inf \left\{ \|g\|_{\mathcal{L}^{\Phi,s}(\mathbb{R})} : G = \mathcal{P}_S(g) \right\}.
\]

We consider
\[
T_G(F) := \lim_{y \to 0^{+}} \int_{\mathbb{R}} F(x+iy)\,\overline{G(x+iy)}\,dx.
\]

Using the Poisson representations, we write
\[
\int_{\mathbb{R}} F(x+iy)\,\overline{G(x+iy)}\,dx
= \int_{\mathbb{R}} (P_y * f)(x)\, \overline{(P_y * g)(x)}\, dx.
\]
By the convolution identity for the Poisson kernel, we have
\[
\int_{\mathbb{R}} (P_y * f)(x)\, \overline{(P_y * g)(x)}\, dx
= \int_{\mathbb{R}} (P_{2y} * g)(x)\, \overline{g(x)}\, dx.
\]
Setting $t = 2y$, we obtain
\[
T_G(F)
= \lim_{t \to 0} \int_{\mathbb{R}} (P_t * f)(x)\, \overline{g(x)}\, dx.
\]
Using the duality between $H^{\Phi}(\mathbb{R})$ and $\mathcal{L}^{\Phi,s}(\mathbb{R})$, we have
\[
\left| \int_{\mathbb{R}} f(x)\, \overline{(P_t * g)(x)}\, dx
- \int_{\mathbb{R}} f(x)\, \overline{g(x)}\, dx \right|
= |\langle f, P_t * g - g \rangle|
\le C \, \|(P_t * f)-f\|_{H^{\Phi}(\mathbb{R})}^{\mathrm{lux}} \, \|g\|_{\mathcal{L}^{\Phi,s}(\mathbb{R})}.
\]
Since $P_t * f \to f$ in $H^{\Phi}(\mathbb{R})$ as $t \to 0^{+}$, it follows that
\[
\lim_{t \to 0^{+}} \int_{\mathbb{R}} (P_t * f)(x)\, \overline{g(x)}\, dx
= \int_{\mathbb{R}} f(x)\, \overline{g(x)}\, dx.
\]

Therefore,
\[
T_G(F) = \int_{\mathbb{R}} f(x)\, \overline{g(x)}\, dx.
\]

Moreover, using again the duality estimate,
\[
|T_G(F)| = |\langle f, g \rangle|
\le C \, \|f\|_{H^{\Phi}(\mathbb{R})}^{\mathrm{lux}} \, \|g\|_{\mathcal{L}^{\Phi,s}(\mathbb{R})}
\lesssim \|F\|_{H^{\Phi}(\mathbb{C}_{+})}^{\mathrm{lux}} \, \|G\|_{\mathcal{H}\mathcal{L}^{\Phi,s}(\mathbb{C}_{+})}.
\]

The proof is complete.
\end{proof}

\proof[Proof of Theorem \ref{pro:main2aaop}.]
Let   $T\in \left(H^{\Phi}(\mathbb{C}_{+})\right)^{*}$ and put    
 $$ \widetilde{T}(f)= T(\mathcal{P}_S(f)), \forall f \in H^{\Phi}(\mathbb{R}).    $$
By construction, $\widetilde{T}\in \left(H^{\Phi}(\mathbb{R})\right)^{*}$. Indeed,  
$$ | \widetilde{T}(f) | = | T(\mathcal{P}_S(f))| \leq \|T\|_{\left(H^{\Phi}(\mathbb{C}_{+})\right)^{*}} \|\mathcal{P}_S(f)\|_{H^{\Phi}}^{lux} \leq \|T\|_{\left(H^{\Phi}(\mathbb{C}_{+})\right)^{*}}\|f\|_{H^{\Phi}}^{lux},  $$
according to Theorem \ref{pro:main5aq5}. Since $\left(H^{\Phi}(\mathbb{R})\right)^{*}= \mathcal{L}^{\Phi,s}(\mathbb{R})$, according to Theorem \ref{pro:main2aaoamdp}, there is a unique function $ g\in \mathcal{L}^{\Phi,s}(\mathbb{R})$ such that $$ \widetilde{T}(f)= \langle f , g \rangle := \int_{\mathbb{R}}f(x)\overline{g(x)}dx, ~~ \forall~f\in H^{\Phi}(\mathbb{R}).   $$
For  $F\in H^{\Phi}(\mathbb{C}_{+})$, there exists a  function $f \in H^{\Phi}(\mathbb{R})$ such that  $$   F(x+iy)=(P_y * f)(x)=\mathcal{P}_S(f)(x+iy), ~~ \forall~ x+iy \in \mathbb{C}_{+},   $$
according to Theorem \ref{pro:main6mmq4p}.  We deduce that,
$$  T(F) = T(\mathcal{P}_S(f))= \widetilde{T}(f)=\int_{\mathbb{R}}f(x)\overline{g(x)}dx,   $$
where $\langle \cdot , \cdot \rangle$ denotes the duality pairing. It follows that
$$  T(F) = \lim_{y \to 0} \int_{\mathbb{R}} F(x+iy)\,\overline{\mathcal{P}_S(g)(x+iy)}\,dx,   $$
according to Proposition \ref{pro:main2aaqaop}. 
Therefore, $T$ is represented by $\mathcal{P}_S(g) \in \mathcal{H}\mathcal{L}^{\Phi,s}(\mathbb{C}_{+})$, which shows that
\[
\left(H^{\Phi}(\mathbb{C}_{+})\right)^{*}
\subset \mathcal{H}\mathcal{L}^{\Phi,s}(\mathbb{C}_{+}).
\]

The reverse inclusion follows from the continuity of the pairing, and we conclude that
\[
\left(H^{\Phi}(\mathbb{C}_{+})\right)^{*}
= \mathcal{H}\mathcal{L}^{\Phi,s}(\mathbb{C}_{+}).
\]
\epf

\subsection{Proof of  Theorem \ref{pro:mainfqmaqppaaqq5}.} 

We start by recalling the following duality result.
\begin{theorem}[Theorem 2.6, \cite{djefeuto}]\label{pro:main 5aqk3pl}
Let $\Phi \in \mathscr{U} \cap \nabla_{2}$ and let $\Psi$ be its complementary function. The topological dual of $H^{\Phi}(\mathbb{C}_{+})$,   $\left(H^{\Phi}(\mathbb{C}_{+})\right)^{*}$    is isomorphic to $H^{\Psi}(\mathbb{C}_{+})$, in the sense that, for all  $T\in \left(H^{\Phi}(\mathbb{C}_{+})\right)^{*}$, there is a unique $G\in H^{\Psi}(\mathbb{C}_{+})$ such that
 $$   \langle F,G  \rangle =T(F) := \lim_{y\to 0}\int_{\mathbb{R}}F(x+iy)\overline{G(x+iy)}dx, ~~ \forall~F\in H^{\Phi}(\mathbb{C}_{+}). 
   $$
Moreover,    
$$  \|G\|_{H^{\Psi}}^{lux} \approx \sup\{ |\langle F,G  \rangle| : F\in H^{\Phi}(\mathbb{C}_{+})~~ \text{with}~~\|F\|_{H^{\Phi}}^{lux} \leq 1 \}. $$   
\end{theorem}

Let us now prove Theorem \ref{pro:mainfqmaqppaaqq5}.
\proof[Proof of Theorem \ref{pro:mainfqmaqppaaqq5}.]
The proof relies on a combination of three key ingredients: 
factorization in Hardy--Orlicz spaces, duality, and product estimates. 
These tools allow us to characterize the boundedness of Hankel operators.

\medskip

\noindent

\textit{(i)} 
Let $\Phi_{4}$ be defined by
\[
\Phi_{4}^{-1}(t)=\Phi_{1}^{-1}(t)\Psi_{2}^{-1}(t), \qquad t>0,
\]
where $\Psi_{2}$ is the complementary function of $\Phi_{2}$. By Lemma \ref{pro:mainfqmaqpq5}, $\Phi_{4}\in \mathscr{U}\cap\nabla_{2}$. Moreover,
\[
\Phi_{4}^{-1}(t)\Phi_{3}^{-1}(t)
= \Phi_{1}^{-1}(t)\Psi_{2}^{-1}(t)\frac{\Phi_{2}^{-1}(t)}{\Phi_{1}^{-1}(t)}
= \Phi_{2}^{-1}(t)\Psi_{2}^{-1}(t)\sim t,
\]
so that $\Phi_{3}$ is the complementary function of $\Phi_{4}$. Hence, by Theorem \ref{pro:main 5aqk3pl}, 
\[
\big(H^{\Phi_{4}}(\mathbb{C}_{+})\big)^{*}=H^{\Phi_{3}}(\mathbb{C}_{+}).
\]

\smallskip

We now prove the equivalence using a unified argument. For $F$ and $G$ in the appropriate Hardy--Orlicz spaces with $\|G\|\leq 1$, we use
\[
|\langle h_{b}(F),G\rangle|
=|\langle b,FG\rangle|.
\]

If $b\in H^{\Phi_{3}}(\mathbb{C}_{+})$, then by duality and product estimates,
\[
|\langle b,FG\rangle|
\lesssim \|b\|_{H^{\Phi_{3}}}^{lux}
\|F\|_{H^{\Phi_{1}}}^{lux}
\|G\|_{H^{\Psi_{2}}}^{lux},
\]
which yields
\begin{equation}\label{eq:ineitedehay}
\|h_{b}(F)\|_{H^{\Phi_{2}}}^{lux}
\lesssim \|b\|_{H^{\Phi_{3}}}^{lux}
\|F\|_{H^{\Phi_{1}}}^{lux}.
\end{equation}

Conversely, if $h_{b}$ is bounded, then for $F\in H^{\Phi_{4}}(\mathbb{C}_{+})$ with $\|F\|\leq 1$, we factorize $F=F_{1}F_{2}$ as in Theorem \ref{pro:main2aqop}. Then
\[
|\langle b,F\rangle|
=|\langle h_{b}(F_{1}),F_{2}\rangle|
\lesssim \|h_{b}\|,
\]
which implies
\begin{equation}\label{eq:ineitaqedehay}
\|b\|_{H^{\Phi_{3}}}^{lux}\lesssim \|h_{b}\|.
\end{equation}

\medskip

\textit{(ii)} 
Assume $\Phi\in \mathscr{U}\cap\nabla_{2}$, so that
\[
\Phi^{-1}(t)\Psi^{-1}(t)\approx t.
\]
Since $(H^{1})^{*}=BMOA$, the same argument as above gives the equivalence
\[
b\in BMOA(\mathbb{C}_{+})
\quad \Longleftrightarrow \quad
h_{b}:H^{\Phi}\to H^{\Phi} \text{ bounded},
\]
together with
\begin{align}
\|h_{b}(F)\|_{H^{\Phi}}^{lux}
&\lesssim \|b\|_{BMOA}\|F\|_{H^{\Phi}}^{lux}, \label{eq:ineiteday}\\
\|b\|_{BMOA}
&\lesssim \|h_{b}\|. \label{eq:ineitaqedaqehay}
\end{align}

\medskip

\textit{(iii)} 
Using $\Phi_{2}^{-1}(t)\Psi_{2}^{-1}(t)\approx t$, we obtain
\[
\Phi_{3}^{-1}(t)\approx \Phi_{1}^{-1}(t)\Psi_{2}^{-1}(t),
\]
hence $\Phi_{3}\in \mathscr{L}$. By Theorem \ref{pro:main2aaop},
\[
\big(H^{\Phi_{3}}(\mathbb{C}_{+})\big)^{*}
=\mathcal{H}\mathcal{L}^{\Phi_{3},s}(\mathbb{C}_{+}).
\]

Repeating the previous duality argument, we obtain
\begin{align}
\|h_{b}(F)\|_{H^{\Phi_{2}}}^{lux}
&\lesssim \|b\|_{\mathcal{H}\mathcal{L}^{\Phi_{3},s}}
\|F\|_{H^{\Phi_{1}}}^{lux}, \label{eq:ineitedPay}\\
\|b\|_{\mathcal{H}\mathcal{L}^{\Phi_{3},s}}
&\lesssim \|h_{b}\|. \label{eq:ineitahay}
\end{align}

\medskip

\noindent
This completes the characterization of bounded Hankel operators in the Hardy--Orlicz setting.
\epf

\bibliographystyle{plain}

\begin{thebibliography}{1}
 
\bibitem{bansahsehba}
 	\textsc{J.S. Bansah and B.F. Sehba}, \emph{Boundeness of a family of Hilbert-Type operators and its Bergman-Type analogue}, lllinois Journal of Mathematics. Vol(\textbf{59}), (2015),  pp. 949-977.
 

\bibitem{BoGre}
    \textsc{A. Bonami and S. Grellier}, \emph{Hankel operators and weak factorization for Hardy-Orlicz spaces}, Colloquium Mathematicum, (2010), \textbf{118}, pp. 107-132. 

\bibitem{BoGreseh}
    \textsc{A. Bonami, S. Grellier and B. Sehba}, \emph{Boundedness of Hankel operators on $H^{1}(\mathbb{B}_{n})$}, C. R. Math. Acad. Sci. Paris\textbf{ 344}(\textbf{12}) (2007), 749-752.


 
\bibitem{BIJZin}
   \textsc{A. Bonami, T. Iwaniec, P. Jones and M. Zinsmeister}, \emph{On the product of Functions in
   $BMO$ and $H^{1}$}, Ann. Inst. Fourier \textbf{57} (2007), 1405-1439. 

\bibitem{BoKy}
    \textsc{A. Bonami and L.D. Ky}, \emph{Factorization of some Hardy type spaces of holomorphic functions}, arXiv:1406.5438v2 [math.CA] 9 Apr 2015

   
\bibitem{BPSyme}
    \textsc{A. Bonami, M. M. Peloso and F. Symesak}, \emph{Powers of the Szego kernel and Hankel operators
    on Hardy spaces}, Mich. Math. J. \textbf{46} (1999), 225-250.
 

\bibitem{BoSehb}
    \textsc{A. Bonami and B. Sehba}, \emph{Hankel operators between Hardy-Orlicz spaces and products of holomorphic functions}, Revista de la union matematica Argentina, Volumen \textbf{50}, Numero \textbf{2}, (2009), Paginas 187-199


\bibitem{Chen}
   \textsc{S. Chen}, \emph{Geometry of Orlicz Spaces. Institute of mathematics}, polish academy of Sciences, (1996).


\bibitem{CRW}
   \textsc{R. Coifman, R. Rochberg and G. Weiss}, \emph{Factorization theorems for Hardy spaces in several variables}, Ann. Math. \textbf{103} (1976), 611-635.
 
 \bibitem{yangLiangKy}\textsc{Dachun Yang, Yiyu Liang, Luong Ky}
   \emph{Real-Variable theory of Musielack-Orlicz Hardy spaces}, Springer (2017).  
   
\bibitem{djefeuto}
       \textsc{J.M. Tanoh Dje and J. Feuto}, \emph{Applied and Numerical Harmonic Analysis: Representation theorems in
       Hardy-Orlicz spaces on the upper complex
       half-plane}. Birkhäuser (Springer Nature Switzerland AG), Suisse (Juillet 2024),  https://doi.org/10.1007/978-3-031-66375-8-2
       
       
       
       
   
\bibitem{djefeuto1}
       \textsc{Dje, JM.T., Feuto, J.}, \emph{Factorization of Hardy-Orlicz Space on the Disk and applications to Hankel Operators}. Complex Anal. Oper. Theory 20, 7 (2026). https://doi.org/10.1007/s11785-025-01843-y
      
    
 
 \bibitem{djesehb}
      \textsc{J.M. Tanoh Dje and B.F. Sehba}, \emph{Carleson embeddings for Hardy-Orlicz and Bergman-Orlicz spaces of the upper half-plane}. Funct. Approx. Comment. Math. 64 (2021), no. 2, 163–201.
 
 
 \bibitem{djesehb1}
      \textsc{J.M. Tanoh Dje and B.F. Sehba}, \emph{Carleson embeddings and pointwise multipliers between  Hardy-Orlicz and Bergman-Orlicz spaces of the upper half-plane}. New York J. Math. 31 (2025) 1271–1315
 
\bibitem{Pduren2}\textsc{P.L. Duren,}
 \emph{Theory of $H^{p}$ Spaces}, Academic Press, (1970).
 
 

 
\bibitem{Jbgarnett}
 	\textsc{J.B. Garnett}, \emph{Bounded Analytic functions}, Academic Press, Inc, Springer (2007).
 
 
 \bibitem{FGolse}
  	\textsc{F. Golse}, \emph{Distributions, analyse de Fourier, \'equations aux D\'eriv\'es partielles},  (Octobre 2012).
  	
 	
 	
\bibitem{GrePelo}
     \textsc{S. Grellier and M.M. Pelesso}, \emph{Decomposition Theorems for Hardys spaces on convex domains of finite type}, Illinois Journal of Mathematics, Volume \textbf{46}, Number \textbf{1}, Spring (2002), Pages 207-232
 	
 	

\bibitem{kokokrbec}
   \textsc{V. Kokilashvili, M. Krbec}, \emph{Weighted Inequalities in Lorentz and Orlicz spaces},World Scientific publishing. C.O.Pte.Ltd (1991).

	
\bibitem{javadmas}\textsc{J. Mashreghi}, \emph{Representation th\'eorems in Hardy spaces}, Cambridge University Press, New York, (2009).



\bibitem{raoren}
 	\textsc{M.M. Rao and Z.D. Ren}, \emph{Theory of Orlicz Spaces}, Marcel Dekker,INC, \textbf{270}.(1991).

\bibitem{wrudin}
 	\textsc{W. Rudin}, \emph{Analyse r\'eelle et complexe}, MASSON et $C^{ie}$, PARIS. (1975). 


\bibitem{sehbaedgc}
     \textsc{B.F. Sehba and E. Tchoundja}, \emph{Duality for large Bergman-Orlicz spaces Boundedness of Hankel operators},Complex variables and Elliptic Equations, Vol(\textbf{62}), February 2017.


\bibitem{sehbaedgc1}
     \textsc{B.F. Sehba and E. Tchoundja}, \emph{Hankel operators on holomorphic Hardy-Orlicz spaces}, Integral Equations and Operator Theory, Volume (\textbf{73})(2012), pages 331-349.



 \bibitem{Jan}\textsc{J. Szajkowski,}
  \emph{Modular spaces of analytic functions in the half-plane. I}. Functiones et Approximatio * XIII* 1982* UAM.  pp. 39-53.
  
 
\bibitem{Viviani}\textsc{B. E. Viviani,}
  \emph{An atomic decomposition of the predual of $BMO_{\rho}$}, Rev. Mat. Iberoamericana \textbf{3}(3-4) (1987), 401-425.
 
 
\bibitem{voltiko}\textsc{A.L. Volberg and V.A. Tolokonnikov,}
  \emph{Hankel operators and problems of best
  approximation of unbounded functions}. Translated from Zapiski Nauchnykh Seminarov Leningradskogo Matematicheskogo Instituta im. V.A. Steklova AN SSSR, \textbf{141} (1985), 5-17.
  

\end{thebibliography}

\end{document}